\documentclass{amsart}

\usepackage[margin=1in]{geometry}
\usepackage{amssymb}
\usepackage{hyperref}
\usepackage{txfonts}
\usepackage{amsmath,amscd}
\usepackage[all]{xy}

\newtheorem{theorem}{Theorem}[section]
\newtheorem{lemma}[theorem]{Lemma}
\newtheorem{proposition}[theorem]{Proposition}
\newtheorem{corollary}[theorem]{Corollary}
\newtheorem{conjecture}[theorem]{Conjecture}

\theoremstyle{definition}
\newtheorem{definition}[theorem]{Definition}
\newtheorem{example}[theorem]{Example}

\theoremstyle{remark}
\newtheorem{remark}[theorem]{Remark}

\numberwithin{equation}{section}

\def\cha{{\rm char\,}} 
 \def\CC{{\mathbb C}} \def\FF{{\mathbb F}}
\def\Sym{\mathrm{Sym}} \def\QSym{\mathrm{QSym}} \def\NSym{\mathbf{NSym}}
\def\Comp{\mathfrak{Comp}}
\def\H{\mathcal{H}} \def\II{\mathcal I}\def\SS{\mathfrak{S}}
\def\P{\mathbf{P}} \def\C{\mathbf{C}} 
 
\def\bq{\mathbf{q}}    \def\bs{{\mathbf s}} 
\def\bw{\mathbf{w}}
 
\def\htimes{\,\hat\otimes\,}
  \def\rad{\operatorname{rad}}

\begin{document}

\title{Hecke algebras with independent parameters}
\author{Jia Huang}
\address{Department of Mathematics and Statistics, University of Nebraska at Kearney, Kearney, Nebraska, USA}
\curraddr{}
\email{huangj2@unk.edu}
\thanks{The author is grateful to Pasha Pylyavskyy and Victor Reiner for asking inspiring questions which lead to this work. He also thanks Victor Reiner for partial support from NSF grant DMS-1001933.
}

\begin{abstract}
We study the Hecke algebra $\H(\bq)$ over an arbitrary field $\FF$ of a Coxeter system $(W,S)$ with independent parameters $\bq=(q_s\in\FF:s\in S)$ for all generators. This algebra is always linearly spanned by elements indexed by the Coxeter group $W$. This spanning set is indeed a basis if and only if every pair of generators joined by an odd edge in the Coxeter diagram receive the same parameter. In general, the dimension of $\H(\bq)$ could be as small as $1$. We construct a basis for $\H(\bq)$ when $(W,S)$ is simply laced. We also characterize when $\H(\bq)$ is commutative, which happens only if the Coxeter diagram of $(W,S)$ is simply laced and bipartite.   In particular, for type A we obtain a tower of semisimple commutative algebras whose dimensions are the Fibonacci numbers. We show that the representation theory of these algebras has some features in analogy/connection with the representation theory of the symmetric groups and the 0-Hecke algebras.
\end{abstract}

\maketitle

\section{Introduction}

Let $W:=\langle S:(st)^{m_{st}}=1,\ \forall s,t\in S \rangle$ be a Coxeter group. The \emph{(Iwahori-)Hecke algebra} of the Coxeter system $(W,S)$ is a one-parameter deformation of the group algebra of $W$, which has significance in many areas, such as algebraic combinatorics, knot theory, quantum groups, representation theory of $p$-adic groups, and so on. 
We generalize the definition of the Hecke algebra of $(W,S)$ from a single parameter to multiple independent parameters.

\begin{definition}\label{def:Hq}
Let $\FF$ be an arbitrary field. The \emph{Hecke algebra $\H(\bq)=\H_S(\mathbf q)$ of the Coxeter system $(W,S)$ with independent parameters $\bq=(q_s\in\FF:s\in S)$} is the (associative) $\FF$-algebra generated by $\{T_s:s\in S\}$ with 
\begin{itemize}
\item
quadratic relations $(T_s-1)(T_s+q_s)=0$ for all $s\in S$,
\item
braid relations $(T_sT_tT_s\cdots)_{m_{st}} = (T_tT_sT_t\cdots)_{m_{st}}$ for all $s,t\in S$.\end{itemize}
Here $(aba\cdots)_m$ is an alternating product of $m$ terms.
\end{definition}

The algebra $\H(\bq)$ can be represented by the Coxeter diagram of $(W,S)$ with extra labels $q_s$ for all vertices $s\in S$. For simplicity we only draw the labels of the vertices but not  the vertices themselves. For example, we draw
\[
\xymatrix @C=8pt{
1 \ar@{=}[r] & 0 \ar@{-}[r] & 1\ar@{-}[r] & 0\ar@{-}[r] & 1\ar@{-}[r] & 0 \ar@{-}[r] & 1 \ar@{-}[r] & 0
}
\]
for the usual Coxeter system of type $B_8$ whose Coxeter diagram is
\[
\xymatrix @C=8pt{
s_1 \ar@{=}[r] & s_2 \ar@{-}[r] & s_3 \ar@{-}[r] & s_4 \ar@{-}[r] & s_5 \ar@{-}[r] & s_6 \ar@{-}[r] & s_7 \ar@{-}[r] & s_8
}
\]
with independent parameters $\bq=(q_{s_i}:1\leq i\leq 8)=(1,0,1,0,1,0,1,0)$.

The quadratic relations for $\H(\bq)$ can be rewritten as $T_s^2=(1-q_s)T_s+q_s$ for all $s\in S$. If $q_s\ne0$ then $T_s$ is invertible and $T_s^{-1}=q_s^{-1}T_s+1-q_s^{-1}$. For any $w\in W$ with a reduced expression $w=st\cdots r$ where $s,t,\ldots,r\in S$, the element $T_w:=T_sT_t\cdots T_r$ is well defined thanks to the word property of $W$ (see e.g. \cite[Theorem~3.3.1]{BjornerBrenti}). 

If $q_s=q$ for all $s\in S$ then $\H(\bq)$ is the usual Hecke algebra of $(W,S)$ with parameter $q$. If one only insists $q_s=q_t$ whenever $m_{st}$ is odd, then $\H(\bq)$ is 
the \emph{Hecke algebra with unequal parameters} in the sense of Lusztig~\cite{Lusztig}. 
Now we allow $\bq=(q_s\in\FF:s\in S)$ to be arbitrary. The following result may be well known to the experts, and we include a proof for it in the end of Section~\ref{sec:collapse} for completeness.

\begin{theorem}\label{thm:TwBasis}
The algebra $\H(\bq)$ is always spanned by $\{T_w:w\in W\}$, which is indeed a basis if and only if $\H(\bq)$ is a Hecke algebra with unequal parameters, i.e. $q_s=q_t$ whenever $m_{st}$ is odd.
\end{theorem}

In general, we show that the algebra $\H(\bq)$ could be much smaller than the group algebra $\FF W$. 

\begin{theorem}\label{thm:collapse}
If there exist $s,t\in S$ with $m_{st}$ odd such that $q_s$ and $q_t$ are distinct nonzero parameters, then one has $\H_S(\bq) \cong \H_{S\setminus R}(\bq)$ where $R$ consists of all elements $r\in S$ connected to $s$ via some path with odd edge weights and nonzero vertex labels in the Coxeter diagram of $(W,S)$. 
\end{theorem}

Thus we always assume without loss of generality that $\H(\bq)$ is \emph{collapse-free}, i.e. if $m_{st}$ is odd and $q_s\ne q_t$ then at least one of $q_s$ and $q_t$ is $0$. We next characterize when $\H(\bq)$ is commutative. 

\begin{theorem}\label{thm:commut}
The algebra $\H(\bq)$ is collapse-free and commutative if and only if $(W,S)$ is simply laced and exactly one of $q_s$ and $q_t$ is $0$ for any pair of elements $s,t\in S$ with $m_{st}=3$.
\end{theorem}

We construct a basis for $\H(\bq)$ (not necessarily commutative) when $(W,S)$ is simply laced (Theorem~\ref{thm:SimplyLaced}). It implies the dimension a commutative $\H(\bq)$, 
giving one motivation for our study of the commutative case.

\begin{corollary}\label{cor:DimFib}
Let $G$ be the underlying graph of the Coxeter diagram of $(W,S)$, and let $\II(G)$ be the set of all independent sets in $G$. If $\H(\bq)$ is collapse-free and commutative then its dimension is $|\II(G)|$ (the Merrifield-Simmons index of the graph $G$). In particular, if $(W,S)$ is of type $A_n$ then the dimension of $\H(\bq)$ is the Fibonacci number $F_{n+2}$.
\end{corollary}

\begin{example}
Let $\FF$ be a field with at least $3$ distinct elements 0, 1, and c. Let $\H(\bq)$ be given by the diagram below.
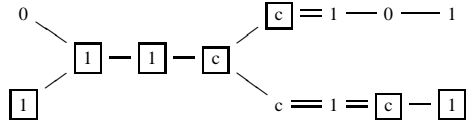
\[ \scriptsize
\xymatrix @R=0pt @C=8pt {
0 \ar@{-}[rd] & & & & \framebox{c}\ar@{=}[r] & 1\ar@{-}[r] & 0\ar@{-}[r] & 1\\
 & \framebox{1} \ar@{-}[r] & \framebox{1}\ar@{-}[r] & \framebox{c} \ar@{-}[rd] \ar@{-}[ru] \\ 
\framebox{1} \ar@{-}[ru] & & & & {\rm c} \ar@{=}[r] & 1\ar@{=}[r] & \framebox{c}\ar@{-}[r] & \framebox{1} \\
}
\]
Removing the boxed elements gives $3$ connected components $0$, $\xymatrix @C=8pt{ {\rm c} \ar@{=}[r] & 1,}$ and $\xymatrix @C=8pt{ 1 \ar@{-}[r] & 0 \ar@{-}[r] & 1. }$ Thus the dimension of $\H(\bq)$ is $2\cdot 8\cdot 5 =80$ by Theorem~\ref{thm:collapse}, Theorem~\ref{thm:commut}, and Corollary~\ref{cor:DimFib}. 
\end{example}

Theorem~\ref{thm:commut} shows that if $\H(\bq)$ is collapse-free and commutative then the Coxeter diagram of $(W,S)$ must be a simply laced bipartite graph. Computations in {\sf Magma} suggest the following conjecture, which is verified for type A (Theorem~\ref{thm:MinDimTypeA}). This gives another motivation for our study of the commutative case.

\begin{conjecture}
If the Coxeter diagram of $(W,S)$ is a simply laced bipartite graph $G$, then a collapse-free $\H(\bq)$ has  minimum dimension equal to $|\II(G)|$, which is attained when $\H(\bq)$ is commutative.
\end{conjecture}

For the irreducible simply laced Coxeter systems of type $A$, $D$, $\widetilde A$, and $\widetilde D$, the dimensions of collapse-free and commutative Hecke algebras $H(\bq)$ are given below, which all happen to satisfy the Fibonacci recurrence.

\begin{center}
\begin{tabular}{|c||c|c|c|}
\hline
Coxeter diagram & Dimensions  & Known as & OEIS entry \\
\hline
$A_n$ ($n\geq1$) & 2,3,5,8,13,\ldots & Fibonacci numbers $F_{n+2}$ & A000045 \\
\hline
$D_n$ ($n\geq2$)& 4,5,9,14,23,\ldots & ? & A000285 \\
\hline
\raisebox{-3pt}{$\widetilde A_n$ ($n\geq 3$)} & \raisebox{-2pt}{4,7,11,18,29,\ldots} & \raisebox{-2pt}{Lucas numbers $L_n$} & \raisebox{-2pt}{A000032} \\
\hline
\raisebox{-3pt}{$\widetilde D_n$ ($n\geq 5$)} & \raisebox{-2pt}{17, 24, 41,65,106,\ldots} & \raisebox{-2pt}{?} & \raisebox{-2pt}{A190996}\\	
\hline
\end{tabular}
\end{center}

Note that the Coxeter diagram of $\widetilde A_n$ is a cycle of length $n$, which is bipartite if and only if $n$ if even. However, the dimensions given above for $\widetilde A_n$ make sense for all integers $n\geq1$. This is because we can define a commutative algebra $\H(G,R)$ whose dimension is $|\II(G)|$ for any (unweighted) simple graph $G$ with vertex set $V(G)$ and edge set $E(G)$ and for any $R\subseteq V(G)$, such that a collapse-free and commutative Hecke algebra $\H(\bq)$ is isomorphic to $\H(G,R)$ where $G$ is the Coxeter diagram of the simply laced $(W,S)$ and $R=\{s\in S:q_s=-1\}$. This algebra $\H(G,R)$ is defined as the quotient of the polynomial algebra $\FF[x_v:v\in V(G)]$ by its ideal generated by
\[
\{x_r^2: r\in R\} \cup
\{x_v^2-x_v: v\in V(G)\setminus R\} \cup
\{x_ux_v: uv\in E(G)\}.
\]
It is also a quotient of the \emph{Stanley-Reisner ring of the independence complex of $G$}~\cite{CookNagel}. 

We show the following results on the representation theory of $\H(G,R)$. The projective indecomposable $\H(G,R)$-modules are indexed by $\II(G-R)$, where $G-R$ is the graph obtained form $G$ by deleting $R$ and all edges incident to $R$. The simple $\H(G,R)$-modules are all one-dimensional and also indexed by $\II(G-R)$. The Cartan matrix of $\H(G,R)$ is a diagonal matrix. The algebra $\H(G,R)$ is semisimple if and only if $R=\emptyset$.

We next apply the above results to type $A$. Let $G=P_{n-1}$ be a path with $n-1$ vertices. One sees that the dimension of the algebra $\H(P_{n-1},R)$ is equal to the Fibonacci number $F_{n+1}$. We further assume that this algebra is semisimple, i.e. $R=\emptyset$, and write $\H_n:=\H(P_{n-1},\emptyset)$. If $\cha(\FF)\ne2$ then $\H_n$ is isomorphic to the Hecke algebra $\H(\bq)$ of the Coxeter system of type $A_{n-1}$ with independent parameters $\bq=(0,1,0,1,\ldots)$ or $\bq=(1,0,1,0,\ldots)$. We summarize our results on the algebra $\H_n$ below. The reader who is familiar with the representation theory of the symmetric group $\SS_n$ and/or the 0-Hecke algebra $\H_n(0)$ can see certain features of our results in analogy with $\SS_n$ and/or $H_n(0)$.

The semisimple commutative algebra $\H_n$ has $F_{n+1}$ many non-isomorphic simple modules, which are all one-dimensional and indexed by compositions of $n$ with internal parts larger than $1$. The \emph{Grothendieck group} $G_0(\H_n)$ of finite dimensional representations of $\H_n$ is a free abelian group on these simple $\H_n$-modules. The tower of algebras $\H_\bullet: \H_0\hookrightarrow \H_1\hookrightarrow \H_2\hookrightarrow \cdots$ has a \emph{Grothendieck group}
\[
G_0(\H_\bullet):=\bigoplus_{n\geq0} G_0(\H_n)
\] 
with a product and a coproduct given by the induction and restriction along the embeddings $\H_m\otimes\H_n\hookrightarrow\H_{m+n}$.

Although \emph{not} a bialgebra, $G_0(\H_\bullet)$ has a self-dual basis consisting of simple $\H_n$-modules for all $n\geq0$. We provide explicit formulas for the structure constants of the product and coproduct of $G_0(\H_\bullet)$ in terms of this self-dual basis, which are naturally all positive. This result connects $G_0(\H_\bullet)$ to the Grothendieck groups of the finite dimensional (projective) representations of the 0-Hecke algebras $\H_n(0)$, or equivalently, the dual Hopf algebras $\NSym$ of \emph{noncommutative symmetric functions} and $\QSym$ of \emph{quasisymmetric functions}. It turns out that $G_0(\H_\bullet)$ is a quotient algebra of $\NSym$ and a subcoalgebra of $\QSym$, but its antipode satisfies a different rule than the antipodes of $\QSym$ and $\NSym$. The \emph{Bratteli diagram} of the tower $\H_\bullet$ is a binary tree on compositions with internal parts larger than $1$.

This paper is structured as follows. We first provide preliminaries in Section~\ref{sec:pre}. Then we discuss when $\H(\bq)$ collapses or becomes commutative in Section~\ref{sec:collapse}. We study the algebra $\H(\bq)$ of a simply laced Coxeter system in Section~\ref{sec:simply-laced}, and investigate the simply laced bipartite case in Section~\ref{sec:bipartite}. We provide more results on the commutative case in Section~\ref{sec:H(G,R)}, and give the type A specialization in Section~\ref{sec:H01}. Finally we give remarks and questions in Section~\ref{sec:future}.

\section{Preliminaries}\label{sec:pre}

\subsection{Coxeter groups and Hecke algebras}\label{sec:Coxeter}
A \emph{Coxeter group} is a group with the following presentation
\[
W:=\langle S:s^2=1,\ (sts\cdots)_{m_{st}}=
(tst\cdots)_{m_{st}},\ \forall s,t\in S,\ s\ne t \rangle
\]
where the generating set $S$ is finite, $m_{st}\in\{2,3,\ldots\}\cup\{\infty\}$, and $(aba\cdots)_m$ is an alternating product of $m$ terms. By convention no relation is imposed between $s$ and $t$ if $m_{st}=\infty$. The pair $(W,S)$ is called a \emph{Coxeter system}.

The Coxeter diagram of $(W,S)$ is an edge-weighted graph whose vertices are the elements in $S$ and whose edges are the unordered pairs $\{s,t\}$ with weight $m_{st}$ for all $s,t\in S$ such that $m_{st}\geq 3$, $s\ne t$. An edge with weight $m_{st}<\infty$ is often drawn as $m_{st}-2$ many multiple edges between $s$ and $t$. An edge is \emph{simply laced} if its weight is $3$. If every edge is simply laced then the Coxeter system $(W,S)$ and its Coxeter diagram are both called \emph{simply laced}. 

An element $w$ in $W$ can be written as a product of elements in $S$. Among all such expressions the shortest ones are called \emph{reduced}, and the length of a reduced expression of $w$ is called the \emph{length} of $w$ and denoted by $\ell(w)$. A \emph{nil-move} deletes $s^2$ and a \emph{braid-move} replaces $(sts\cdots)_{m_{st}}$ with $(tst\cdots)_{m_{st}}$ in the expressions of $w\in W$ as products of elements in $S$. By \cite[Theorem~3.3.1]{BjornerBrenti}, $W$ satisfies the following word property.

\vskip3pt\noindent\textbf{Word Property.}  
\emph{Any expression of $w\in W$ as a product of elements in $S$ can be transformed into a reduced expression of $w$ by braid-moves and nil-moves, and every pair of reduced expressions for $w$ can be connected via braid-moves.}
\vskip3pt

A subset $I\subseteq S$ generates a \emph{parabolic subgroup} $W_I:=\langle I\rangle$ of $W$. The pair $(W_I,I)$ is a Coxeter system whose Coxeter diagram is the edge-weighted subgraph of the Coxeter diagram of $(W,S)$ induced by the vertex subset $I\subseteq S$. If $S_1,\ldots,S_k$ are the vertex sets of the connected components of the Coxeter diagram of $(W,S)$ then $W = W_{S_1}\times\cdots\times W_{S_k}$. Thus $(W,S)$ is \emph{irreducible} if its Coxeter diagram is connected. 

There is a well known classification for finite irreducible Coxeter groups, among which type A is of particular interest. The symmetric group $\SS_n$ is the Coxeter group of type $A_{n-1}$ with generating set $S$ consisting of the adjacent transpositions $s_i:=(i,i+1)$ for $i=1,\ldots,n-1$. The Coxeter diagram of $\SS_n$ is the path $\xymatrix @C=9pt {s_1 \ar@{-}[r] & s_2 \ar@{-}[r] & \cdots \ar@{-}[r] & s_{n-1}}$.
  
The \emph{(Iwahori-)Hecke algebra} $\H_S(q)$ of a Coxeter system $(W,S)$ is a one-parameter deformation of the group algebra of $W$. Let $\FF$ be a field and let $q\in\FF$. Then $\H_S(q)$ is defined as the $\FF$-algebra generated by $\{T_s:s\in S\}$ with 
\begin{itemize}
\item
quadratic relations: $(T_s-1)(T_s+q) = 1,\ \forall s\in S$,
\item
braid relations: $(T_sT_tT_s\cdots)_{m_{st}} = (T_tT_sT_t\cdots)_{m_{st}},\ \forall s, t\in S,\ s\ne t$.
\end{itemize}
The specialization of the Hecke algebra $\H_S(q)$ at $q=1$ gives the group algebra $\FF W$, and the specialization at $q=0$ gives the \emph{0-Hecke algebra} $\H_S(0)$. If $(W,S)$ is of type $A_{n-1}$ then we write $\H_n(q):=\H_S(q)$ and $\H_n(0):=\H_S(0)$.

If $w\in W$ has a reduced expression $w=st\cdots r$, where $s,t,\ldots,r\in S$, then $T_w:=T_sT_t\cdots T_r$ is well defined thanks to the word property of $W$. It is well known that $\{T_w:w\in W\}$ is a basis for $\H_S(q)$. One has
\begin{equation}\label{eq:reg}
T_s T_w = \begin{cases}
(1-q)T_w+qT_{sw}, & \ell(sw)<\ell(w), \\
T_{sw}, & \ell(sw)>\ell(w),
\end{cases}
\end{equation}
for all $s\in S$ and $w\in W$. This gives the \emph{regular representation} of $\H_S(q)$.

\subsection{Representation theory of associative algebras}
We review some general results on the representation theory of associative algebras; see e.g.~\cite[\S I]{ASS}. Let $\FF$ be a field and let $A$ be a finite dimensional (unital associative) $\FF$-algebra. Let $M$ be a (left) $A$-module. If $M$ has no submodules except $0$ and itself then $M$ is \emph{simple}. If $M$ is a direct sum of simple $A$-modules then $M$ is \emph{semisimple}. The algebra $A$ is \emph{semisimple} if it is semisimple as an $A$-module. Every module over a semisimple algebra is also semisimple. If $M$ cannot be written as a direct sum of two nonzero $A$-submodules, then $M$ is \emph{indecomposable}. If $M$ is a direct summand of a free $A$-module, then $M$ is \emph{projective}. 

The \emph{(Jacobson) radical} ${\rm rad}(M)$ of $M$ is the intersection of all maximal $A$-submodules of $M$, which turns out to be the smallest submodule $N$ of $M$ such that $M/N$ is semisimple. One has
$
\rad(M_1\oplus M_2) = {\rm rad}(M_1)\oplus\rad(M_2)
$
if $M_1$ and $M_2$ are two $A$-modules. The radical of the algebra $A$ is defined as $\rad(A)$ with $A$ itself viewed as an $A$-module. If $A$ happens to be commutative then all nilpotent elements in $A$ form an ideal of $A$, called the \emph{nilradical} of $A$, which is always contained in $\rad(A)$. 
The \emph{top} of $M$ is the quotient module ${\rm top}(M):=M/{\rm rad}(M)$. The \emph{socle} ${\rm soc}(M)$ of $M$ is the sum of all minimal submodules of $M$, which is the largest semisimple submodule of $M$.

Every $A$-module can be written as a direct sum of indecomposable $A$-submodules. Let $A$ itself as an $A$-module be a direct sum of indecomposable $A$-modules $\P_1,\ldots,\P_k$. Although $\P_i$ is not simple in general, its top $\C_i$ is. Moreover, every projective indecomposable $A$-module is isomorphic to some $\P_i$, and every simple $A$-module is isomorphic to some $\C_i$. Suppose without loss of generality that $\{\P_1,\ldots,\P_\ell\}$ and $\{\C_1,\ldots,\C_\ell\}$ are complete lists of non-isomorphic projective indecomposable $A$-modules and simple $A$-modules, respectively, where $\ell\le k$. Then the \emph{Cartan matrix} of $A$ is $[a_{ij}]_{i,j\in[\ell]}$ where $a_{ij}$ is the multiplicity of $\C_j$ among the composition factors of $\P_i$.

The \emph{Grothendieck group $G_0(A)$ of the category of finitely generated $A$-modules} is defined as the abelian group $F/R$, where $F$ is the free abelian group on the isomorphism classes $[M]$ of finitely generated $A$-modules $M$, and $R$ is the subgroup of $F$ generated by the elements $[M]-[L]-[N]$ corresponding to all exact sequences $0\to L\to M\to N\to0$ of finitely generated $A$-modules. The \emph{Grothendieck group $K_0(A)$ of the category of finitely generated projective $A$-modules} is defined similarly. We often identify a finitely generated (projective) $A$-module with the corresponding element in the Grothendieck group $G_0(A)$ ($K_0(A)$). It turns out that $G_0(A)$ and $K_0(A)$ are free abelian groups with bases $\{\C_1,\ldots,\C_\ell\}$ and $\{\P_1,\ldots,\P_\ell\}$, respectively. If $L,M,N$ are all projective $A$-modules, then the exact sequence $0\to L\to M\to N\to0$ is equivalent to the direct sum decomposition $M\cong L\oplus N$. If $A$ is semisimple then $G_0(A)=K_0(A)$ since $\P_i=\C_i$ for all $i$.

Let $B$ be a subalgebra of $A$. The \emph{induction} $N\uparrow\,_B^A$ of a $B$-module $N$ from $B$ to $A$ is the $A$-module $A\otimes_B N$. The \emph{restriction} $M\downarrow\,_B^A$ of an $A$-module $M$ from $A$ to $B$ is $M$ itself viewed as a $B$-module. The induction and restriction are well defined for isomorphic classes of modules.



\subsection{Representation theory of symmetric groups and 0-Hecke algebras}\label{sec:rep}

The (complex) representation theory of the symmetric group is fascinating and has rich connections with symmetric function theory. The simple $\CC\SS_n$-modules $S_\lambda$ are indexed by partitions $\lambda$ of $n$, and every $\CC\SS_n$-module is a direct sum of simple $\CC\SS_n$-modules, i.e. $\CC\SS_n$ is semisimple. Thus the Grothendieck group $G_0(\CC\SS_n)=K_0(\CC\SS_n)$ is a free abelian group on the isomorphism classes $[S_\lambda]$ for all partitions $\lambda$ of $n$. The tower of groups $\SS_\bullet: \SS_0\hookrightarrow \SS_1\hookrightarrow \SS_2\hookrightarrow \cdots$ has a Grothendieck group 	
\[
G_0(\CC\SS_\bullet) := \bigoplus_{n\geq0} G_0(\CC\SS_n).
\]
Using the natural embedding $\SS_m\times\SS_n\hookrightarrow \SS_{m+n}$, one can define the product of $S_\mu$ and $S_\nu$ as the induction of $S_\mu\otimes S_\nu$ from $\SS_m\times \SS_n$ to $\SS_{m+n}$ for all partitions $\mu\vdash m$ and $\nu\vdash n$, and define the coproduct of $S_\lambda$ as the sum of its restriction to $\SS_i\times\SS_{n-i}$ for $i=0,1,\ldots,n$, for all partitions $\lambda\vdash n$. This gives $G_0(\CC\SS_\bullet)$ a self-dual graded Hopf algebra structure, as the product and coproduct share the same structure constants, namely the \emph{Littlewood-Richardson coefficients}.

The \emph{Frobenius characteristic map} ch sends a simple $S_\lambda$ to the Schur function $s_\lambda$, giving a Hopf algebra isomorphism between the Grothendieck group $G_0(\CC\SS_\bullet)$ and $\Sym$, the \emph{ring of symmetric functions} (see Stanley~\cite[Chapter 7]{EC2}). 

The 0-Hecke algebra $\H_n(0)$ has analogous representation theory as the symmetric group $\SS_n$.  We first review some notation. A \emph{composition} is a sequence $\alpha=(\alpha_1,\ldots,\alpha_\ell)$ of positive integers. Let $\sigma_i:=\alpha_1+\cdots+\alpha_i$ for $i=1,\ldots,\ell$. The \emph{size} $|\alpha|$ of the composition $\alpha$ is the sum of all its \emph{parts} $\alpha_1,\ldots,\alpha_\ell$, i.e. $|\alpha|=\sigma_\ell$. If $|\alpha|=n$ then we say that $\alpha$ is a composition of $n$ and write $\alpha\models n$. The \emph{descent set} of $\alpha$ is $D(\alpha):=\{\sigma_1,\ldots,\sigma_{\ell-1}\}$. Sending $\alpha$ to $D(\alpha)$ gives a bijection between compositions of $n$ and subsets of $[n-1]$.  

Now recall from Norton~\cite{Norton} that the 0-Hecke algebra $\H_n(0)$ has the following decomposition
\[
\H_n(0)=\bigoplus_{\alpha\models n} \P_\alpha(0)
\]
where the $\P_\alpha(0)$'s are pairwise non-isomorphic indecomposable $\H_n(0)$-modules. The top of $\P_\alpha(0)$ is one-dimensional and denoted by $\C_\alpha(0)$. Thus the two Grothendieck groups $G_0(\H_n(0))$ and $K_0(\H_n(0))$ are free abelian groups on the isomorphism classes of $\C_\alpha(0)$ and $\P_\alpha(0)$, respectively, for all compositions $\alpha$. Associated with the tower of algebras $\H_\bullet(0): \H_0(0)\hookrightarrow \H_1(0)\hookrightarrow \H_2(0)\hookrightarrow \cdots$ are two Grothendieck groups
\[
G_0(\H_\bullet(0)):=\bigoplus_{n\geq0}G_0(\H_n(0)) \quad\text{and}\quad
K_0(\H_\bullet(0)):=\bigoplus_{n\geq0}K_0(\H_n(0)).
\]
They are dual graded Hopf algebras with product and coproduct again given by induction and restriction of representations along the natural embeddings $\H_m(0)\otimes \H_n(0)\hookrightarrow \H_{m+n}(0)$ of algebras. The duality is given by the pairing 
$\langle \P_\alpha(0),\C_\beta(0) \rangle := 
\delta_{\alpha,\,\beta}$
for all compositions $\alpha$ and $\beta$.


For later use we review the explicit formulas for the product of $K_0(\H_\bullet(0))$ and the coproduct of $G_0(\H_\bullet(0))$. Let $\alpha=(\alpha_1,\ldots,\alpha_\ell)$ and $\beta=(\beta_1,\ldots,\beta_k)$ be compositions of $m$ and $n$, respectively. We write 
\[
\alpha\beta:=(\alpha_1,\ldots,\alpha_\ell,\beta_1,\ldots,\beta_k) \quad\text{and}\quad
\alpha\rhd\beta := (\alpha_1,\ldots,\alpha_{\ell-1},\alpha_\ell+\beta_1,\beta_2,\ldots,\beta_k).
\]
For any $i\in\{0,1,\ldots,m\}$, let $r$ be the largest integer such that $\sigma_r:=\alpha_1+\cdots+\alpha_r$ is no more than $i$, and write
\[
\alpha_{\leq i}:=(\alpha_1,\ldots,\alpha_r,i-\sigma_r) \quad\text{and}\quad
\alpha_{>i}:=(\sigma_{r+1}-i,\alpha_{r+2},\ldots,\alpha_\ell)
\]
where we ignore $i-\sigma_r$ if it happens to be $0$. 

\begin{proposition}[Krob and Thibon~\cite{KrobThibon}]\label{prop:ProdP}
For any $\alpha\models m$ and $\beta\models n$ one has
\[
\P_\alpha(0) \htimes \P_\beta(0) := \left( \P_\alpha(0) \otimes \P_\beta(0) \right) \uparrow\,_{\H_m(0)\otimes \H_n(0)}^{\H_{m+n}(0)} = 
\P_{\alpha\beta}(0) \oplus \P_{\alpha \rhd \beta}(0),
\]
\[
\Delta(\C_\alpha(0)) := \sum_{i=0}^m \C_\alpha(0) \downarrow\,_{\H_i(0)\otimes \H_{m-i}(0)}^{\H_m(0)} = \sum_{i=0}^m \C_{\alpha_{\leq i}}(0)  \otimes \C_{\alpha_{>i}}(0).
\]
\end{proposition}

For example, one has $\P_{213}(0)\htimes\P_{223}(0) = \P_{213223}(0) \oplus \P_{21523}(0)$. Let $\emptyset$ be the empty composition of $n=0$. Then
\[
\Delta(\C_{121}(0)) = \C_\emptyset(0)\otimes\C_{121}(0) + \C_1(0)\otimes\C_{21}(0) + \C_{11}(0)\otimes\C_{11}(0) + \C_{12}(0)\otimes\C_{1}(0) + \C_{121}(0) \otimes \C_\emptyset(0).
\]

The representation theory of the 0-Hecke algebras is connected with the dual graded Hopf algebras $\QSym$ of \emph{quasisymmetric functions} and $\NSym$ of \emph{noncommutative symmetric functions}. There are dual bases for $\QSym$ and $\NSym$ consisting of the \emph{fundamental quasisymmetric functions} $F_\alpha$ and the \emph{noncommutative ribbon Schur functions} $\bs_\alpha$ for all compositions $\alpha$. Krob and Thibon~\cite{KrobThibon} introduced two Hopf algebra isomorphisms
\[
\mathrm{Ch}: G_0(\H_\bullet(0))\cong\QSym \quad\text{and}\quad 
\mathbf{ch}: K_0(\H_\bullet(0))\cong \NSym
\]
defined by $\mathrm{Ch}(\C_\alpha(0))=F_\alpha$ and $\mathbf{ch}(\P_\alpha(0))=\bs_\alpha$ for all compositions $\alpha$. There is an injection $\Sym\hookrightarrow \QSym$ of Hopf algebras given by inclusion, as well as a surjection $\NSym\twoheadrightarrow\Sym$ of Hopf algebras by taking commutative image.

\section{Collapse and commutativity}\label{sec:collapse}

Let $(W,S)$ be a Coxeter system and let $\FF$ be a field. In this section we study when the Hecke algebra $\H(\bq)=\H_S(\bq)$ of $(W,S)$ with independent parameters $\bq=(q_s\in \FF:s\in S)$ collapses or becomes commutative. 

We first study the \emph{parabolic subalgebras} of $\H(\bq)$. We know that any subset $R\subseteq S$ generates a Coxeter subsystem $(W_R,R)$ of $(W,S)$. However, the subalgebra of $\H(\bq)$ generated by $\{T_r:r\in R\}$ is not necessarily isomorphic to the Hecke algebra $\H_R(\bq)$ of the Coxeter system $(W_R,R)$ with independent parameters $(q_r:r\in R)$. For example, if there exist two elements $s$ and $t$ in $S$ such that $q_s$ and $q_t$ are distinct nonzero parameters and $m_{st}$ is odd, then the algebra $\H_{\{s\}} (\bq)$ is $2$-dimensional, but Theorem~\ref{thm:12} below gives $T_s=1$ in $\H(\bq)$. To guarantee an isomorphism between these two algebras we assume that $R\subseteq S$ is \emph{admissible}, i.e. if $m_{st}$ is odd for $s\in R$ and $t\in S\setminus R$ then either $q_s=0$ or $q_t=0$. If $R$ is admissible then one sees that $S\setminus R$ is also admissible. We denote the generating set of $\H_R(\bq)$ by $\{T'_r:r\in R\}$, which satisfies the relations $(T'_r-1)(T'_r+q)=0$ and $(T'_rT'_tT'_r\cdots)_{m_{rt}} = (T'_tT'_rT'_t\cdots)_{m_{rt}}$ for all $r,t\in R.$

\begin{proposition}\label{prop:subalgebra}
For any $R\subseteq S$ there is an algebra surjection from $\H_R(\bq)$ to the subalgebra of $\H(\bq)$ generated by $\{T_r:r\in R\}$ by sending $T'_r$ to $T_r$ for all $r\in R$, which is an isomorphism when $R$ is admissible.
\end{proposition}

\begin{proof}
Sending $T'_r$ to $T_r$ for all $r\in R$ gives an algebra map  $\phi:H_R(\bq)\to \H(\bq)$ whose image is the subalgebra of $\H(\bq)$ generated by $\{T_r:r\in R\}$. Suppose that $R$ is admissible and define 
\[
\psi(T_s)=
\begin{cases}
T'_s, & {\rm if}\  s\in R,\\
1, & {\rm if}\ s\in S\setminus R,\ q_s\ne0,\\
0, & {\rm if}\ s\in S\setminus R,\ q_s=0.
\end{cases}
\]
One sees that the quadratic relations are preserved by $\psi$. We next check the braid relations. Let $s,t\in S$ with $m_{st}=m$.

If $s$ and $t$ are both in $R$ then $\psi(T_s)=T'_s$ and $\psi(T_t)=T'_t$ satisfy the same braid relation as $T_s$ and $T_t$.

If $s\in R$ and $t\in S\setminus R$, then $\psi(T_t)\in\{0,1\}$. When $m$ is even one has
\[
(\psi(T_s)\psi(T_t)\psi(T_s)\cdots)_m=(\psi(T_t)\psi(T_s)\psi(T_t)\cdots)_m.
\]
When $m$ is odd and $q_t=0$ one has $\psi(T_t)=0$ and the above quality still holds. When $m$ is odd and $q_t\ne 0$ one has $\psi(T_t)=1$ and the admissibility of $R$ implies $q_s=0$. Thus
\[
(\psi(T_s)\psi(T_t)\psi(T_s)\cdots)_m=(T'_s)^{(m+1)/2}=(T'_s)^{(m-1)/2}=(\psi(T_t)\psi(T_s)\psi(T_t)\cdots)_m.
\]

It follows that $\psi$ is a well defined algebra map. Restricted to the image of $\phi$, the map $\psi$ is nothing but the inverse of $\phi$. Thus the result holds.
\end{proof}

We say that a path in the Coxeter diagram of $(W,S)$ is \emph{odd} if all its edges have odd weights, and \emph{nonzero} if all its vertices, including the two end vertices, correspond to nonzero parameters. The \emph{collapsed subset} of $S$ consists of all elements $r\in S$ that are connected to some other vertex $s$ (depend on $r$) with $q_s\ne q_r$ via an odd nonzero path. 

\begin{theorem}\label{thm:12}
If $R$ is the collapsed subset of $S$ then (i) $T_r=1$, $\forall r\in R$, (ii) $T_s\notin\FF$, $\forall s\in S\setminus R$, and (iii) $\H(\bq)\cong \H_{S\setminus R}(\bq)$.
\end{theorem}
 
\begin{proof}
By definition, for any $r\in R$ there exists an odd nonzero path $(r, s, \ldots,t)$ from $r$ to some $t\in S$ such that $q_r\ne q_t$. We show (i) by induction on the length of the path. First assume that the length is $1$, i.e. there is an edge between $r$ and $t$ with an odd weight $m:=m_{rt}$. The braid relation between $T_r$ and $T_t$ implies that
\[
T_r(T_rT_tT_r\cdots T_r)_m = (T_rT_tT_r\cdots T_t)_{m+1} = (T_tT_rT_t\cdots T_t)_m T_t.
\]
Using the quadratic relations for $T_r$ and $T_t$ one obtains
\[
q_r(T_tT_rT_t\cdots)_{m-1}+(1-q_r)(T_rT_tT_r\cdots)_m = q_t(T_tT_rT_t\cdots)_{m-1} + (1-q_t) (T_tT_rT_t\cdots)_m.
\]
Hence 
\[
(q_r-q_t)(T_tT_rT_t\cdots T_r)_{m-1} = (q_r-q_t)(T_rT_tT_r\cdots T_r)_m = (q_r-q_t)(T_tT_rT_t\cdots T_t)_m.
\] 
Since $q_r\ne0$, $q_t\ne0$, and $q_r\ne q_t$, one can apply the inverses of $T_r$, $T_t$, and $(q_r-q_t)$ to get $T_r=T_t=1$.

Now suppose that the path $(r,s,\ldots,t)$ has length at least two. If $q_r\ne q_s$ then $T_r=1$ by the above argument. Otherwise $q_r=q_s\ne q_t$ and one has $T_{s}=1$ by induction, since $(s,\ldots,t)$ is an odd nonzero path of smaller length. Then applying $T_r^{-1}$ to the braid relation between $T_r$ and $T_{s}$ gives $T_r=1$. This proves (i).

To show (ii), we assume $T_s\in\FF$ for some $s\in S$. If $q_s=0$ then $\{s\}$ is admissible and thus the subalgebra of $\H(\bq)$ generated by $T_s$ is $2$-dimensional by Proposition~\ref{prop:subalgebra}, which is absurd. Therefore $q_s\ne0$. Let $U$ be the set of all elements in $S$ that are connected to $s$ via odd nonzero paths, including $s$ itself. Then $q_u\ne0$ for all $u\in U$. One sees that $U$ is admissible and hence the subalgebra of $\H(\bq)$ generated by $\{T_u:u\in U\}$ is isomorphic to the algebra $\H_U(\bq)$ by Proposition~\ref{prop:subalgebra}. If $|\{q_u:u\in U\}|=1$ then $\H_R(\bq)$ has a basis indexed by $W_U$, and hence $T_s\notin\FF$, a contradiction. Therefore $|\{q_u:u\in U\}|\geq 2$. This forces $s\in R$ and establishes (ii). 

Finally, one sees that $S\setminus R$ is admissible. By Proposition~\ref{prop:subalgebra}, $\H_{S\setminus R}(\bq)$ is isomorphic to the subalgebra of $\H(\bq)$ generated by $\{T_s:s\in S\setminus R\}$. Hence (iii) follows from (i).
\end{proof}

By Theorem~\ref{thm:12}, we may always assume without loss of generality that $\H(\bq)$ is \emph{collapse-free}, i.e. if $m_{st}$ is odd and $q_s\ne q_t$ then either $q_s$ or $q_t$ is $0$. We next develop some lemmas in order to characterize when $\H(\bq)$ is commutative.

\begin{lemma}\label{lem:H(0,q)}
If $S=\{s,t\}$, $q_s=0\ne q_t$, and $m:=m_{st}$ is odd, then $\H(\bq)$ has dimension $2m-3$ and a basis 
\[
\{(T_sT_tT_s\cdots)_k,\ (T_tT_sT_t\cdots)_k: k=0,1,2,\ldots,m-2\}.
\]
\end{lemma}

\begin{proof}
Since $q_s=0\ne q_t$ and $m$ is odd, it follows from the defining relations for $\H(\bq)$ that
\[
(T_sT_tT_s\cdots T_s)_m=(T_sT_tT_s\cdots T_t)_{m+1}=(T_tT_sT_t\cdots T_t)_mT_t=q_t(T_tT_sT_t\cdots)_{m-1}+(1-q_t)(T_tT_sT_t\cdots)_m
\]
which implies $(T_tT_sT_t\cdots)_{m-1}=(T_tT_sT_t\cdots)_m$ and thus $(T_sT_tT_s\cdots)_{m-2}=(T_sT_tT_s\cdots)_{m-1}$. Similarly, 
\[
(T_sT_tT_s\cdots)_m = (T_tT_sT_t\cdots T_s)_{m+1}= T_t(T_tT_sT_t\cdots)_m=q_t (T_sT_tT_s\cdots)_{m-1}+(1-q_t)(T_tT_sT_t\cdots)_m.
\]
Thus $(T_sT_tT_s\cdots T_t)_{m-1}=(T_tT_sT_t\cdots T_t)_m$ and $(T_sT_tT_s\cdots)_{m-2}=(T_tT_sT_t\cdots)_{m-1}$. It follows that $\H(\bq)$ is spanned by the desired basis. Then it remains to show that the dimension of $\H(\bq)$ is at least $2m-3$.

To achieve this, we define an $\H(\bq)$-action on the $\FF$-span of 
$
Z:=\{(sts\cdots)_k,\ (tst\cdots)_k: k=0,1,2,\ldots,m-2\}
$
where $(sts\cdots)_0=(tst\cdots)_0=1$ by convention. The dimension of $\FF Z$ is by definition $|Z|=2m-3$. Define
\[
\begin{cases}
T_s(tst\cdots)_k = (sts\cdots)_{k+1}, & 0\leq k\leq m-3, \\
T_t(sts\cdots)_k=(tst\cdots)_{k+1}, & 0\leq k\leq m-3, \\
T_s(sts\cdots)_k=(sts\cdots)_k, & 1\leq k\leq m-2,\\
T_t(tst\cdots)_k = q_t(sts\cdots)_{k-1} + (1-q_t)(tst\cdots)_k, & 1\leq k\leq m-2,\\
T_s(tst\cdots)_{m-2}=T_t(sts\cdots)_{m-2} = (sts\cdots)_{m-2}.
\end{cases}
\]
One sees that the quadratic relations for $T_s$ and $T_t$ are both satisfied by this action, and so is the braid relation because
\[
(T_sT_tT_s\cdots)_m (z) = (T_tT_sT_t\cdots)_m (z) = (sts\cdots)_{m-2}, \quad\forall z\in Z.
\] 
Hence $\FF Z$ becomes a cyclic $\H(\bq)$-module generated by $1$. This forces the dimension of $\H(\bq)$ to be at least $2m-3$.
\end{proof}

\begin{lemma}\label{lem:01=0}
Suppose that there exists a path $(s=s_0,s_1,s_2,\ldots,s_k=t)$ consisting of simply laced edges in the Coxeter diagram of $(W,S)$, where $k\geq1$. If $q_{s_i}\ne0$ and $m_{ss_i}\leq 3$ for all $i\in[k]$, and $q_s=0$, then $T_sT_t=T_tT_s=T_s$.
\end{lemma}

\begin{proof}
We show $T_sT_t=T_tT_s=T_s$ by induction on $k$. If $k=1$ then 
\[
T_sT_tT_s = T_tT_sT_tT_s = T_t^2T_sT_t = q_t T_sT_t + (1-q_t) T_tT_sT_t.
\]
Since $q_t\ne0$, one has $T_sT_t = T_tT_sT_t$ and thus $T_s=T_tT_s$. Then $T_sT_t = T_sT_tT_s = T_s^2 =T_s$.

Now assume $k\geq2$. If $m_{st}=3$ then $T_sT_t=T_tT_s=T_s$ by the above argument. Assume $m_{st}=2$, i.e. $T_sT_t=T_tT_s$. Let $r=s_{k-1}$. Then $T_rT_s=T_sT_r=T_s$ by induction hypothesis. Thus
\[
T_tT_s = T_sT_rT_tT_r = T_sT_tT_rT_t = T_t^2T_s = q_tT_s + (1-q_t) T_tT_s.
\]
This implies $T_sT_t=T_tT_s = T_s$ which completes the proof.
\end{proof}

Now we provide a characterization for when $\H(\bq)$ is commutative. It implies that there exists $\bq\in\FF^S$ such that $\H(\bq)$ is collapse-free and commutative if and only if the Coxeter diagram of $(W,S)$ is simply laced and bipartite.

\begin{theorem}\label{thm:commutative}
Suppose that $\H(\bq)$ is collapse-free. Then $\H(\bq)$ is commutative if and only if the Coxeter diagram of $(W,S)$ is simply laced and exactly one of $q_s,q_t$ is $0$ for any pair of elements $s,t\in S$ with $m_{st}=3$.
\end{theorem}

\begin{proof}
We first assume that $\H(\bq)$ is commutative. Let $s,t\in S$ with $m_{st}\geq3$. We need to show that $m_{st}=3$ and exactly one of $q_s$ and $q_t$ is $0$. To attain this we first show that $\{s,t\}$ is admissible. By symmetry, it suffices to show that $q_rq_s=0$ for any $r\in S\setminus \{s,t\}$ with $m_{rs}$ odd. Suppose to the contrary that $q_rq_s\ne0$. Then $q_r=q_s$ since $\H(\bq)$ is collapse-free. Let $R$ be a maximal subset of $S$ containing $s$ such that $q_a=q_b$ whenever $a,b\in R$ and $m_{ab}$ is odd. Then $r\in R$. The maximality forces $R$ to be admissible. By Proposition~\ref{prop:subalgebra}, $\H_R(\bq)$ is isomorphic to a subalgebra of $\H(\bq)$ and thus commutative. It also has a basis $\{T_w:w\in W_R\}$ by Theorem~\ref{thm:TwBasis}. Hence $m_{rs}\leq 2$, a contradiction.

Therefore $\{s,t\}$ is admissible. Then $\H_{\{s,t\}}(\bq)$ is isomorphic to a subalgebra of $\H(\bq)$, and hence commutative. Since $m_{st}\geq 3$, Theorem~\ref{thm:TwBasis} implies that $m_{st}$ is odd and $q_s\ne q_t$. Then exactly one of $q_s$ and $q_t$ must be $0$ since $\H(\bq)$ is collapse-free. By Lemma~\ref{lem:H(0,q)}, the dimension of $\H_{\{s,t\}}(\bq)$ is $2m-3$ and hence $m_{st}=3$. This proves one direction of the theorem. The other direction follows from Lemma~\ref{lem:01=0}. 
\end{proof}

Finally, using the results in this section we obtain a proof for Theorem~\ref{thm:TwBasis}. One can check that $\{T_w:w\in W\}$ spans $\H(\bq)$ using the word property of $W$ and the defining relations of $\H(\bq)$. If $q_s=q_t$ whenever $m_{st}$ is odd, then $\{T_w:w\in W\}$ is a basis for $\H(\bq)$ by Lusztig~\cite[Proposition~3.3]{Lusztig}. Conversely, suppose that $\{T_w:w\in W\}$ is a basis for $\H(\bq)$. Let $s,t\in S$ with $m:=m_{st}$ odd. The dimension $d$ of the subalgebra of $\H(\bq)$ generated by $T_s$ and $T_t$ equals the cardinality of the subgroup $\langle s,t\rangle$ of $W$, which is $2m$ by the word property of $W$. On the other hand,  if $q_s\ne q_t$ then either $d=1<2m$ when $q_sq_t\ne0$ by Theorem~\ref{thm:12}, or $d\leq 2m-3<2m$ when $q_sq_t=0$ by Proposition~\ref{prop:subalgebra} and Lemma~\ref{lem:H(0,q)}. Hence $q_s=q_t$.

\section{The simply laced case} \label{sec:simply-laced}

In this section we study a collapse-free Hecke algebra $\H(\bq)$ with independent parameters $\bq=(q_s\in\FF: s\in S)$ of a simply laced Coxeter system $(W,S)$. We first give some lemmas in order to construct a basis for $\H(\bq)$.

\begin{lemma}\label{lem:Si}
If $(W,S)$ is simply-laced then $S$ decomposes into a disjoint union of $S_1,\ldots,S_k$ such that 

\noindent(i) the elements of each $S_i$ receive the same parameters and are connected in the Coxeter diagram of $(W,S)$,

\noindent(ii) if $s\in S_i$, $t\in S_j$, $i\ne j$, then either $m_{st}=2$ or exactly one of $q_s$ and $q_t$ is $0$.
\end{lemma}

\begin{proof}
We remove from the Coxeter diagram of $(W,S)$ all the edges whose two end vertices correspond to distinct parameters. Let $S_1,\ldots,S_k$ be the vertex sets of the connected components of the resulting graph. 

If $s,t\in S_i$ then there exists a path from $s$ to $t$, whose vertices have the same parameter. Thus (i) holds.

If $s\in S_i$, $t\in S_j$, $i\ne j$, and $m_{st}=3$, then one has $q_s\ne q_t$ and thus exactly one of $q_s$ and $q_t$ is $0$ since $\H(\bq)$ is collapse-free. Hence (ii) holds.
\end{proof}

Let $W_i:=\langle S_i\rangle$ for all $i=1,\ldots,k$. We say an element $w_i\in W_i$ \emph{dominates} $S_j$ if $i\ne j$ and there exist $s\in S_i$ and $t\in S_j$ such that $q_s=0$, $m_{st}=3$, and $s$ occurs in some reduced expression of $w_i$. Let $W(\bq)$ be the set of all elements $(w_1,\ldots,w_k)\in W_1\times\cdots\times W_k$ such that $w_j=1$ whenever some $w_i$ dominates $S_j$. We need to define an $\H(\bq)$-action on $\FF W(\bq)$. Let $s$ be an arbitrary element in $S$. Then $s\in S_i$ for some $i\in [k]$. Let $\bw=(w_1,\ldots,w_k)\in W(\bq)$. We define $T_s(\bw):=(T_s(\bw)_1,\ldots,T_s(\bw)_k)\in\FF W(\bq)$ as follows.

If $S_i$ is dominated by some $w_j$, then $T_s$ acts \emph{trivially} on $\bw$, meaning that $T_s(\bw):=\bw$. Otherwise $T_s$ acts \emph{nontrivially} on $\bw$: if $\ell(sw_i)<\ell(w_i)$ then $T_s(\bw)_i=(1-q)w_i+qsw_i$ and $T_s(\bw)_j=w_j$ for all $j\ne i$; if $\ell(sw_i)>\ell(w_i)$ then $T_s(\bw)_i=sw_i$, $T_s(\bw)_j=1$ for all $j\ne i$ such that $s$ dominates $S_j$, and $T_s(\bw)_j=w_j$ for all $j\ne i$ such that $s$ does not dominates $S_j$. In other words, if $S_i$ is not dominated by $w_j$ for all $j\ne i$ then $T_s$ acts on the $i$th component of $\bw$ in the same way as the regular representation of the Hecke algebra $\H_{S_i}(q_s)$ (see (\ref{eq:reg})), and for all $j\ne i$ one has 
\[
T_s(\bw)_j=
\begin{cases}
w_j, & \textrm{if $s$ does not dominate $S_j$},\\
1, & \textrm{if $s$ dominates $S_j$}.
\end{cases}
\]

\begin{lemma}\label{lem:reg}
One has a well defined $\H(\bq)$-action on $\FF W(\bq)$ such that every element $(w_1,\ldots,w_k)$ in $W(\bq)$ is equal to $T_{w_1}\cdots T_{w_k}(1)$.
\end{lemma} 

\begin{proof}
Let $s\in S_i$ and let $\bw=(w_1,\ldots,w_k)\in W(\bq)$. We first show that $T_s(\bw)\in\FF W(\bq)$. We may assume that $T_s$ acts nontrivially on $\bw$, i.e. $S_i$ is not dominated by $w_j$ for all $j\ne i$. If $\ell(sw_i)<\ell(w_i)$ then $\bw\in W(\bq)$ implies
\[
T_s(\bw)=(1-q)\bw+q(w_1,\ldots,w_{i-1},sw_i,w_{i+1},\ldots,w_k)\in W(\bq).
\]
If $\ell(sw_i)>\ell(w_i)$ then $T_s(\bw)\in W(\bq)$ since $T_s(\bw)_i=sw_i$ and $T_s(\bw)_j=1$ whenever $s$ dominates $S_j$.

Next we verify the quadratic relation for the action of $T_s$. If $T_s$ acts trivially on $\bw$ then $T_s^2=(1-q_s)T_s+q_s$ clearly holds. Assume that $T_s$ acts nontrivially on $\bw$ and apply $T_s$ again to $T_s(\bw)$. For the $i$-th component this is the same as the regular representation of $\H_{S_i}(q_s)$ (see~\ref{eq:reg}). Hence $T_s^2=(1-q_s)T_s+q_s$ holds for the $i$-th component. Let $j\ne i$. If $s$ does not dominates $S_j$ then $T_s(\bw)_j=w_j$ is fixed by $T_s$. If $s$ dominates $S_j$ then $T_s(w_j)=1$ is also fixed by $T_s$, and $q_s=0$. Hence $T_s^2=(1-q_s)T_s+q_s$ also holds for the $j$-th component for all $j\ne i$.

Next we verify the braid relation between $T_s$ and $T_t$ for any $t\in S_i\setminus\{s\}$. If one of $T_s$ and $T_t$ acts trivially on $\bw$ then so does the other. Thus we may assume that $T_s$ and $T_t$ both act nontrivially on $\bw$. Then they both act on the $i$-th component of $\bw$ by the regular representation of $\H_{S_i}(q_s)$ and hence the braid relation holds for this component. Let $j\ne i$ and let $T(s,t)$ be any product of $T_s$ and $T_t$ that contains both of them. If either $s$ or $t$ dominates $S_i$ then $T(s,t)$ sends $w_j$ to $1$. If neither of $s$ and $t$ dominates $S_j$ then $T(s,t)$ fixes $w_j$. Hence the braid relation between $T_s$ and $T_t$ also holds for the $j$-th component for all $j\ne i$.

Next assume that $t\in S_j$ and $i\ne j$. First consider the case when $s$ dominates $S_j$. Since $q_s=0$, one has $T_s(\bw)_i=w_i$ if $\ell(sw_i)<\ell(w_i)$ and $T_s(\bw)_i=sw_i$ if $\ell(sw_i)>\ell(w_i)$. In either case $T_t$ acts trivially on $T_s(\bw)$, i.e. $T_t(T_s(\bw))=T_s(\bw)$. On the other hand, since $q_t\ne0$, one sees that $T_t$ dominates nothing and thus fixes all components of $\bw$ except the $j$-th one. Since $s$ dominates $S_j$, one also has $T_s(T_t(\bw))_j=T_s(\bw)_j=1$. Hence $T_s(T_t(\bw))=T_s(\bw)$.

Similarly if $t$ dominates $S_i$ then one has $T_sT_t(\bw)=T_t(\bw)=T_tT_s(\bw)$. For the remaining case, that is, when $s$ does not dominate $S_j$ and $t$ does not dominates $S_i$, one has $m_{st}=2$ by Lemma~\ref{lem:Si} (ii). We need to show that both actions of $T_sT_t$ and $T_tT_s$ on $\bw$ are the same. One sees for both actions that $T_s$ and $T_t$ act separately on $w_i$ and $w_j$ by the regular representations of $\H_{S_i}(q_s)$ and $\H_{S_j}(q_t)$, respectively. Let $h\in[k]\setminus\{i,j\}$. If $S_h$ is dominated by either $s$ or $t$ then both $T_sT_t$ and $T_tT_s$ sends $w_h$ to $1$. Otherwise both $T_sT_t$ and $T_tT_s$ fixes $w_j$. Hence $T_sT_t(\bw)=T_tT_s(\bw)$.

Therefore one has a well defined action of $\H(\bq)$ on $\FF W(\bq)$. One sees that every element $(w_1,\ldots,w_k)$ in $W(\bq)$ is equal to $T_{w_1}\cdots T_{w_k}(1)$ by induction on $\ell(w_1)+\cdots+\ell(w_k)$. This completes the proof.
\end{proof}

\begin{theorem}\label{thm:SimplyLaced}
Assume that $(W,S)$ is simply-laced and $\H(\bq)$ is collapse-free. Then $\H(\bq)$ has a basis 
\[
B(\bq):=\{T_{w_1}\cdots T_{w_k}:(w_1,\ldots,w_k)\in W(\bq)\}.
\] 
\end{theorem}

\begin{proof}
Theorem~\ref{thm:TwBasis} shows that $\H(\bq)$ is spanned by $\{T_w:w\in W\}$. Let $s\in S_i$, $t\in S_j$, and $i\ne j$. If $m_{st}=2$ then $T_sT_t=T_tT_s$. If $m_{st}=3$ then we may assume $0=q_s\ne q_t$ by Lemma~\ref{lem:Si} and it follows from Lemma~\ref{lem:01=0} that $T_sT_r= T_s = T_rT_s$ for all $r\in S_j$. Hence for any $w\in W$ one can write $T_w=T_{w_1}\cdots T_{w_k}$ where $\bw=(w_1,\ldots,w_k)\in W(\bq)$. This shows that $B(\bq)$ is a spanning set for $\H(\bq)$. On the other hand, it follows from Lemma~\ref{lem:reg} that $B(\bq)$ is also linearly independent. Thus $B(\bq)$ is a basis for $\H(\bq)$.
\end{proof}

\begin{corollary}\label{cor:FiniteDim}
Suppose that $(W,S)$ is simply-laced and let $S_1,\ldots,S_k$ be given by Lemma~\ref{lem:Si}. 

(i) A collapse-free $\H(\bq)$ is finite dimensional if and only if $W_i:=\langle S_i\rangle$ is finite for all $i\in[k]$.

(ii) There exists $\bq\in \FF^S$ such that $\H(\bq)$ is collapse-free and finite dimensional if and only if there exists $R\subseteq S$ such that the parabolic subgroups $\langle R \rangle$ and $\langle S\setminus R\rangle$ are finite. 
\end{corollary}

\begin{proof}
(i) By Theorem~\ref{thm:SimplyLaced}, a collapse-free $\H(\bq)$ is finite dimensional if and only if $W(\bq)$ is finite. For any $i\in[k]$, there are injections $W_i\hookrightarrow W(\bq)\hookrightarrow W_1\times\cdots\times W_k$. Hence $W(\bq)$ is finite if and only if $W_i$ is finite for all $i\in[k]$.

(ii) Suppose that $\H(\bq)$ is collapse-free and finite dimensional. Let $R:=\{s\in S:q_s=0\}$. By Lemma~\ref{lem:Si}, we may assume $R=S_1\cup\cdots \cup S_j$. Then $\langle R\rangle = \langle S_1\rangle\times\cdots\times\langle S_j\rangle$ and $\langle S\setminus R \rangle = \langle S_{j+1}\rangle\times\cdots\times\langle S_k\rangle$ are both finite groups by (i). Conversely, if there exists a subset $R\subseteq S$ such that $\langle R \rangle$ and $\langle S\setminus R\rangle$ are both finite groups, then $\H(\bq)$ is finite dimensional by (i), where $\bq$ is defined by $q_s=0$ for all $s\in R$ and $q_s=1$ for all $s\notin R$.
\end{proof}

\begin{example}
(i) It is well known that the Coxeter group of affine type $A$ is infinite and so is the associated Hecke algebra with a single parameter. However, if one takes some parameters to be $0$ and others to be $1$, the resulting algebra is finite dimensional, since all the $W_i$'s given in the above theorem are of finite type $A$.

(ii) Let the Coxeter diagram of $(W,S)$ be the complete graph $K_5$ with $5$ vertices. Assume that $\H(\bq)$ is collapse-free. There can be at most two different parameters $0$ and $q\ne0$. Both $R:=\{s\in S:q_s=0\}$ and its complement $S\setminus R=\{s\in S:q_s=q\}$ are admissible subsets of $S$, the larger one of which contains at least $3$ elements and thus gives a copy of the infinite dimensional Hecke algebra of affine type $A_3$ with a single parameter as a subalgebra of $\H(\bq)$. Therefore $\H(\bq)$ is never finite dimensional in such cases.
\end{example}

\section{The simply laced bipartite case}\label{sec:bipartite}

By Theorem~\ref{thm:commutative} there exists $\bq\in\FF^S$ such that $\H(\bq)$ is collapse-free and commutative if and only if the Coxeter diagram of $(W,S)$ is simply laced and bipartite. We give more results for such case in this section. Let $T_I:=\prod_{i\in I} T_i$ for all $I\in\II(G)$, where $\II(G)$ consists of independent sets in the underlying graph $G$ of the Coxeter diagram of $(W,S)$.

\begin{corollary}\label{cor:ComDim}
A collapse-free and commutative $\H(\bq)$ has a basis $\{T_I:I\in\II(G)\}$. In particular, if $(W,S)$ is of type $A_n$ then the dimension of $\H(\bq)$ equals the Fibonacci number $F_{n+2}$.
\end{corollary}

\begin{proof}
By Theorem~\ref{thm:commutative}, the Coxeter diagram of $(W,S)$ is a simply laced and bipartite graph $G$ with all edges between the two subsets $\{s\in S:q_s=0\}$ and $\{t\in S:q_t\ne0\}$. Hence the subsets $S_1,\ldots,S_k$ given by Lemma~\ref{lem:Si} are all singleton sets. Then the basis $B(\bq)$ for $\H(\bq)$ given in Theorem~\ref{thm:SimplyLaced} consists of the elements $T_I$ for all $I\in\II(G)$.

Now suppose that $(W,S)$ is of type $A_n$, i.e. its Coxeter diagram is isomorphic to the path $P_n$ with $n$ vertices. If an independent set $I$ in $P_n$ contains one end vertex of $P_n$, then removing this end point from $I$ gives an independent set of $P_{n-2}$; otherwise $I$ is an independent set of $P_{n-1}$. Thus $|\II(P_n)|=|\II(P_{n-1})|+|\II(P_{n-2})|$. One also sees that $|\II(P_i)|=i+1$ if $i=0,1$. Thus $|\II(P_n)|=F_{n+2}$ for all $n\geq0$.
\end{proof}

Computations in {\sf Magma} suggest the following conjecture.

\begin{conjecture}
Suppose that the Coxeter diagram of $(W,S)$ is a simply laced and bipartite graph $G$. The minimum dimension of a collapse-free $\H(\bq)$ is $|\II(G)|$, which is attained when it is commutative.
\end{conjecture}

We will verify this conjecture for type $A_n$. We first need a lemma on the \emph{Fibonacci numbers}, which are defined as $F_0=0$, $F_1=1$, and $F_n=F_{n-1}+F_{n-2}$ for all $n\geq2$.

\begin{lemma}\label{lem:Fib}
If $k\geq 4$ then $k!\geq F_{k+3}+2$. Also, if $a\geq1$ and $b\geq0$ then $F_{a+b}=F_a F_{b+1} + F_{a-1} F_b \leq F_aF_{b+2}$.
\end{lemma}

\begin{proof}
The first result follows easily by induction. It is well known that $F_{a+b}=F_a F_{b+1} + F_{a-1} F_b$ (see Example~\ref{example:Fib1}). Hence $F_{a+b} \leq F_a(F_{b+1}+F_b) =F_aF_{b+2}$.
\end{proof}

\begin{theorem}\label{thm:MinDimTypeA}
Let $\H(\bq)$ be a collapse-free Hecke algebra of type $A_n$ with independent parameters. Then its dimension  is at least the Fibonacci number $F_{n+2}$, and the equality holds if and only if $\H(\bq)$ is commutative. 
\end{theorem}

\begin{proof}
We prove the result by induction on $n$. The Coxeter diagram for type $A_n$ is the path $\xymatrix @C=8pt{ s_1 \ar@{-}[r] & s_2 \ar@{-}[r] & \cdots \ar@{-}[r] & s_n}$. We write $q_i:=q_{s_i}$ for all $i\in[n]$. Let $S_1,\ldots,S_k$ be the subsets of $S$ given by Lemma~\ref{lem:Si}. Then $S_j$ is a path of length $n_j\geq1$ for every $j\in[k]$. We may assume, without loss of generality, that 
\[
S_j=\{s_i: n_1+\cdots+n_{j-1}<i\leq n_1+\cdots+n_j\},\quad\forall j\in[k].
\]

If all parameters in $\bq$ are the same, then $\H(\bq)$ has dimension $(n+1)!\geq F_{n+2}$. Thus we may assume that there exists $j\in[k]$ such that $q_s=q\ne0$ for all $s\in S_j$. Let $a=n_1+\cdots+n_{j-1}$, $b=n_j$, and $c=n_{j+1}+\cdots+n_k$. By convention $a=0$ if $j=1$, and $c=0$ if $j=k$. One sees that $s_a$ and $s_{a+b+1}$ both dominate $S_j$.

By Theorem~\ref{thm:SimplyLaced}, $\H(\bq)$ has dimension $|W(\bq)|$. We need to count the elements $(w_1,\ldots,w_k)$ in $W(\bq)$. If $w_j\ne1$ then any reduced word of $w_{j-1}$ cannot contain $s_a$ and any reduced word of $w_{j+1}$ cannot contain $s_{a+b+1}$. It follows that $(w_1,\ldots,w_{j-1})$ and $(w_{j+1},\ldots,w_k)$ are arbitrary elements in $W(q_{i}:1\leq i\leq a-1)$ and $W(q_{i}:a+b+2\leq i\leq n)$, respectively. Then the number of choices for $(w_1,\ldots,w_k)$ in this case is at least $F_{a+1}((b+1)!-1)F_{c+1}$, by induction hypothesis. Note that this still holds even if $a=0$ or $c=0$, since $F_1=1$.

Similarly, if $w_j=1$ the number of choices for $(w_1,\ldots,w_k)$ is at least $F_{a+2}F_{c+2}$ by induction hypothesis. 

Thus the dimension of $\H(\bq)$ is at least $f(a,b,c):=F_{a+1}((b+1)!-1)F_{c+1}+F_{a+2}F_{c+2}$.
By Lemma~\ref{lem:Fib},
\[
f(a,b,c)=F_{a+1}((b+1)!-2)F_{c+1}+F_{a+c+3}.
\]
If $b=1$ then this becomes $f(a,b,c)=F_{a+c+3}=F_{n+2}$.
If $b=2$ then Lemma~\ref{lem:Fib} implies that
\[
f(a,b,c)> 3F_{a+1}F_{c+1} +F_{a+c+3} \geq F_4 F_{a+c} + F_{n+1} \geq F_n + F_{n+1}=F_{n+2}.
\]
If $b\geq3$ then Lemma~\ref{lem:Fib} implies that 
\[
f(a,b,c)> F_{a+1}F_{b+4}F_{c+1}\geq F_{a+b+3}F_{c+1}\geq F_{n+2}.
\]
Therefore $f(a,b,c)\geq F_{n+2}$ always holds.

Finally, assume $f(a,b,c)=F_{n+2}$. By the above argument, this equality is possible only if $b=1$ and the dimensions of $\H(q_1,\ldots,q_a)$ and $\H(q_{a+2},\ldots,q_n)$ are $F_{a+2}$ and $F_{c+2}$, respectively. Then $\H(q_1,\ldots,q_a)$ and $\H(q_{a+2},\ldots,q_n)$ are commutative by induction hypothesis. The definition for $a$, $b$, and $c$ implies $q_a=0$, $q_{a+1}\ne0$, and $q_{a+2}=0$. It follows from Theorem~\ref{thm:commutative} that $q_i=0$ when $i\equiv a$ mod $2$ and $q_i\ne0$ otherwise. Hence $\H(\bq)$ must be commutative. On the other hand, if $\H(\bq)$ is commutative then its dimension is $F_{n+2}$ by Corollary~\ref{cor:ComDim}. This completes the proof.
\end{proof}

Next we explain the connection between a collapse-free and commutative $\H(\bq)$ and the \emph{M\"obius algebra} $A(L)$ of a finite lattice $L$. According to Stanley~\cite[\S 3.9]{EC1}, the M\"obius algebra $A(L)$ is the monoid algebra of $L$ over $\FF$ with the meet operation, and it is a direct sum of $|L|$ many one-dimensional subalgebras.

Now let $Z$ be a finite rank two poset. Set $X:=\{x\in Z: x>y \textrm{ for some }y\in Z\}$ and $Y=Z\setminus X$. By abuse of notation we denote by $Z$ the underlying graph of $Z$. Let $L$ be the distribute lattice $J(Z)$ of the order ideals of $Z$ ordered by reverse inclusion (so that the meet operation is the union of ideals). Suppose that $(W,S)$ is a Coxeter system whose Coxeter diagram coincides with $Z$. Denote by $\H(Z)$ the Hecke algebra $\H(\bq)$ of $(W,S)$ with parameters $\bq=(q_s:s\in S)$ given by $q_s=0$ for all $s\in X$ and $q_s=1$ for all $s\in Y$.

\begin{proposition}
When $\cha(\FF)\ne2$ the algebra $\H(Z)$ is isomorphic the M\"obius algebra of $J(Z)$. 
\end{proposition}

\begin{proof}
By definition, the algebra $\H(Z)$ is generated by $\{T_x:x\in X\}\cup\{T_y:y\in Y\}$ with relations
\[
\begin{cases}
T_x^2=T_x,\ T_y^2=1,& \forall x\in X,\ \forall y\in Y, \\
T_zT_{z'}=T_{z'}T_z, & \forall z,z'\in Z, \\ 
T_xT_y=T_x, & \textrm{if $x>y$ in $Z$ (by Lemma~\ref{lem:01=0})}.
\end{cases}
\]
One has a basis $\{T_I:I\in\II(Z)\}$ for $\H(Z)$ by Corollary~\ref{cor:ComDim}.

When $\cha(\FF)\ne2$ one can replace the generator $T_y$ with $T'_y:=(T_y+1)/2$, which is now an idempotent, for every $y\in Y$. One checks that all other relations given above remain same. Write $T'_x=T_x$ for all $x\in X$. Then the algebra $H(Z)$ is generated by $\{T'_x:x\in X\} \cup\{T'_y:y\in Y\}$ and has a basis $\{T'_I:I\in\II(Z)\}$ where $T'(I):= \prod_{z\in I} T'_z$.

Any independent set $I$ in $\II(Z)$ is an antichain in $Z$, generating an order ideal $J(I)$ consisting of all elements weakly below some element of $I$. Conversely, an order ideal of $Z$ corresponds to an independent set $I\in\II(Z)$ consisting of all maximal elements in this order ideal. Hence sending $T'(I)$ to the order ideal $J(I)$ for all $I\in\II(Z)$ gives a vector space isomorphism $H(Z) \cong A(J(Z))$. To see this isomorphism preserves multiplications, let $I_1$ and $I_2$ be two elements in $\II(Z)$. Then $T'(I_1)T'(I_2)=T'(I_1\circ I_2)$ where 
$I_1\circ I_2$ is obtained from $I_1\cup I_2$ by removing all the elements that are less than some element of $I_1\cup I_2$. On the other hand, the order ideal $J(I_1)\cup J(I_2)$ has maximal elements given by $I_1\circ I_2$, and thus equals $J(I_1\circ I_2)$. This completes the proof.
\end{proof}

\section{The commutative case}\label{sec:H(G,R)}

By Theorem~\ref{thm:commutative} and Corollary~\ref{cor:ComDim}, if $\H(\bq)$ is collapse-free and commutative, then the Coxeter diagram of $(W,S)$ is simply laced with a bipartite underlying graph $G$, and the dimension of $\H(\bq)$ is $|\II(G)|$. In this section we define and study a more general commutative algebra for any (unweighted) simple graph $G$, whose dimension is still $|\II(G)|$.

\subsection{Basic results}
Let $G$ be a simple graph with vertex set $V(G)$ and edge set $E(G)$, and let $R\subseteq V(G)$. We define an algebra $\H(G,R)$ to be the quotient of the polynomial algebra $\FF[x_v:v\in V(G)]$ by the ideal generated by
\[
\{x_r^2: r\in R\} \cup
\{x_v^2-x_v: v\in V(G)\setminus R\} \cup
\{x_ux_v: uv\in E(G)\}.
\]
The image of $x_v$ in the quotient algebra $\H(G,R)$ is still denoted by $x_v$ for all $v\in V$. This algebra $\H(G,R)$ generalizes the commutative algebra $\H(\bq)$ by the following result.

\begin{proposition}\label{prop:HGR}
If $\H(\bq)$ is collapse-free and commutative then it is isomorphic to $\H(G,R)$ as an algebra, where $G$ is the underlying graph of the Coxeter diagram of $(W,S)$ and $R:=\{s\in S:q_s=-1\}$.
\end{proposition}

\begin{proof}
The algebra $\H(\bq)$ has another generating set $\{x_s:s\in S\}$ given by 
\[
x_s:=
\begin{cases}
T_s, & q_s=0,\\
T_s-1, & q_s=-1, \\
(1-T_s)/(1+q_s),  & {\rm  otherwise}. 
\end{cases}
\]
If $\H(\bq)$ is collapse-free and commutative then one can check that the relations for $\{T_s:s\in S\}$ are equivalent to the relations for $\{x_s:s\in S\}$ in the definition of $\H(G,R)$ using Lemma~\ref{lem:01=0}. Thus the result holds.
\end{proof}

\begin{remark}
(i) The set $R=\{s\in S:q_s=-1\}$ associated with $\H(\bq)$ depends on $\cha(\FF)$. For example, an element $s\in S$ with $q_s=1$ belongs to $R$ if and only if $\cha(\FF)=2$. However, once $R$ is chosen for the algebra $\H(G,R)$, our results on $\H(G,R)$ do not depend on $\cha(\FF)$ any more.

(ii) By Theorem~\ref{thm:commutative}, if $\H(\bq)$ is collapse-free and commutative then $R=\{s\in S:q_s=-1\}$ must be an independent set of $G$. But the commutative algebra $\H(G,R)$ is well defined for any simple graph $G$ and any subset $R\subseteq V(G)$. 

(iii) The \emph{Stanley-Reisner ring of the independence complex of $G$} is defined as the quotient of the polynomial algebra $\FF[y_v:v\in V(G)]$ by the \emph{edge ideal} generated by $(y_uy_v:uv\in E(G))$. See e.g. \cite{CookNagel}. The algebra $\H(G,R)$ is a further quotient of the Stanley-Reisner ring of the independence complex of $G$.
\end{remark}

Now we study the algebra $\H(G,R)$ and our results will naturally apply to the commutative algebra $\H(\bq)$ by Proposition~\ref{prop:HGR}. We first need some notation. For any $U\subseteq V(G)$ we write
\[
X_U:=\prod_{u\in U} x_u\quad {\rm and}\quad X^-_U:=\prod_{u\in U} x^-_u
\]
where $x^-_v:=1-x_v$ for all $v\in V(G)$. One sees that $X_U\ne0$ if and only if $U$ belongs to $\mathcal I(G)$, the set of all independent sets in $G$. We define the \emph{length} of a nonzero monomial $X_I$ to be the cardinality $|I|$ of the independent set $I$. We partially order the nonzero monomials by their lengths. We denote by $N(U)$ the set of all vertices that are adjacent to some vertex $u\in U$ in $G$. We will often identify a subset $U$ of $V(G)$ with the subgraph of $G$ induced by $U$, whose vertex set is $U$ and whose edge set is $\{\{u,v\}\in E(G):u,v\in U\}$. We will also write ``$+$'' and ''$-$'' for set union and difference. For example, we write $G-R$ for the subgraph of $G$ induced by $V(G)-R$, and hence $\II(G-R)$ consists of all independent sets of $G-R$. We give two bases for $\H(G,R)$ in the following proposition, which generalizes Corollary~\ref{cor:ComDim}.

\begin{proposition}\label{prop:basisG}
The algebra $\H(G,R)$ has dimension $|\mathcal I(G)|$ and two bases $\{X_I:I\in\mathcal I(G)\}$ and 
\begin{equation}\label{eq:basisG}
\left\{ X_{I+ J} X^-_{G-R-I}: I\in\II(G-R),\ J\in\II(R-N(I)) \right\}.
\end{equation}
\end{proposition}

\begin{proof}
The defining relations for $\H(G,R)$ immediately imply that it is spanned by $\{X_I:I\subseteq I(G)\}$. Let $\FF \mathcal I(G)$ be the vector space over $\FF$ with a basis $\mathcal I(G)$. We define an action of $\H(G,R)$ on $\FF \II(G)$ by 
\[
x_v(I)=
\begin{cases}
0, & {\rm if}\ v\in I\cap R\ {\rm or}\ I\cup\{v\}\notin\mathcal I(G),\\
I\cup\{v\}, & {\rm otherwise.}
\end{cases}
\]
It is not hard to check that this action satisfies the defining relations for $\H(G,R)$. For any $I\in\mathcal I(G)$, one has $X_I(\emptyset)=I$. This forces the spanning set $\{X_I:I\subseteq I(G)\}$ to be a basis for $\H(G,R)$.

One sees that any independent set of $G$ can be written uniquely as $I+J$ for some $I\in\II(G-R)$ and $J\in \II(R-N(I))$, and the shortest term in $X_{I+J} X^-_{G-R-I}$ is $X_{I+ J}$. Thus (\ref{eq:basisG}) is also a basis for $\H(G)$.
\end{proof}

Let $G'$ be a subgraph of $G$ induced by $V'\subseteq V(G)$, and let $R'=V'\cap R$. The following corollary allows us to study the induction of $\H(G',R')$-modules to $\H(G,R)$ and the restriction of $\H(G,R)$-modules to $\H(G',R')$.

\begin{corollary}\label{cor:EmbedAlgebra}
The subalgebra of $\H(G,R)$ generated by $\{x_v:v\in V'\}$ is isomorphic to $\H(G',R')$.
\end{corollary}

\begin{proof}
There is an injection $\phi: \H(G',R')\hookrightarrow \H(G,R)$ of algebras defined by sending the generators $x'_v$ for $\H(G',R')$ to the generators $x_v$ for $\H(G,R)$ for all $v\in V'$. By Proposition~\ref{prop:basisG}, the algebra $\H(G',R')$ admits a basis consisting of the elements $X'_I:=\prod_{v\in I} x'_v$ for all $I\in\II(G')$. The map $\phi$ sends this basis to the basis $\{X_I:I\in\II(G')\}$ for the subalgebra of $\H(G,R)$ generated by $\{x_v:v\in V'\}$, giving the desired isomorphism.
\end{proof}

\subsection{Projective indecomposable modules and simple modules}
We first decompose the algebra $\H(G,R)$ into a direct sum of indecomposable submodules.


\begin{theorem}\label{thm:decomp}
There is an $\H(G,R)$-module decomposition 
\begin{equation}\label{eq:H(G,R)}
\H(G,R) = \bigoplus_{I\subseteq \II(G-R)}\P_I(G,R)
\end{equation}
where each $\P_I(G,R):=\H(G,R)X_IX^-_{G-R-I}$ is an indecomposable $\H(G,R)$-module with a basis 
\begin{equation}\label{eq:BasisP}
\left\{ X_{I+J} X^-_{G-R-I}: J\in \II(R-N(I)) \right\}
\end{equation}
and hence has dimension $|\II(R-N(I))|$. The top of $\P_I(G,R)$, denoted by $\C_I(G,R)$, is one-dimensional and admits an $\H(G,R)$-action by
\[
x_v = 
\begin{cases}
1, & {\rm if}\ v\in I, \\
0, & {\rm if}\ v\in G-I.
\end{cases}
\]
\end{theorem}

\begin{proof}
Let $I\in\II(G-R)$. Since $x_vx^-_v=0$ for any $v\in G-R-I$, and $x_ux_v=0$ whenever $v\in I$ and $u\in N(v)$, one has
\begin{equation}\label{eq:HP}
X_J( X_I X^-_{G-R-I} ) = 
\begin{cases}
 X_{I+J} X^-_{G-R-I}, & {\rm if}\ J-I\in \II(R-N(I)), \\
 0, & {\rm otherwise}
\end{cases}
\end{equation}
for any $J\in\II(G)$. Hence (\ref{eq:BasisP}) spans $\P_I(G,R)$. By Proposition~\ref{prop:basisG}, $\H(G,R)$ has a basis (\ref{eq:basisG}) which is the union of the spanning sets (\ref{eq:BasisP}) for all $I\in\II(G-R)$. This implies the direct sum decomposition (\ref{eq:H(G,R)}) of $\H(G,R)$ and forces the spanning set (\ref{eq:BasisP}) to be a basis for $\P_I(G,R)$. The dimension of $\P_I(G,R)$ is then clear. 

Now we prove that $\P_I(G,R)$ is indecomposable and find its top. Since $x_r^2=0$ for any $r\in R$, the elements in (\ref{eq:BasisP}) are all nilpotent except $X_IX^-_{G-R-I}$. The span $\mathbf N_I$ of these nilpotent elements is contained in the nilradical of $\H(G,R)$, and hence in the radical of $\P_I(G,R)$. By (\ref{eq:HP}), the quotient $\P_I(G,R)/\mathbf N_I$ is isomorphic to the one-dimensional $\H(G,R)$-module $\C_I(G,R)$. It follows that the radical of $\P_I(G,R)$ equals $\mathbf N_I$, and the top of $\P_I(G,R)$ is isomorphic to $\C_I(G,R)$. Then $\P_I(G,R)$ must be indecomposable as its top is simple.
\end{proof}

By Theorem~\ref{thm:decomp}, $\{\P_I(G,R):I\in\II(G-R)\}$ and $\{\C_I(G,R):I\in\II(G-R)\}$ are complete lists of pairwise-nonisomorphic projective indecomposable $\H(G,R)$-modules and simple $\H(G,R)$-modules, respectively. The proof of Theorem~\ref{thm:decomp} shows that the radical of $\P_I(G,R)$ is spanned by $\{X_{I+J}X^-_{G-R-I}: \emptyset\ne J\in\II(R-N(I))\}$ and hence the radical of $\H(G,R)$ is the ideal generated by $\{x_r:r\in R\}$. This ideal coincides with the nilradical of $\H(G,R)$, showing that $\H(G,R)$ is a \emph{Jacobson ring}. Some other consequences of Theorem~\ref{thm:decomp} are listed below.

\begin{corollary}\label{cor:decomp}
Theorem~\ref{thm:decomp} implies the following results.

\noindent(i) 
The algebra $\H(G,R)$ is semisimple if and only if $R=\emptyset$.

\noindent(ii) 
For any $I\in\II(G-R)$ one has $\P_I(G,R)\cong \H(G,R)\otimes_{\H(G-R,\emptyset)} \C_I(G-R,\emptyset)$.

\noindent(iii) 
The socle of $\P_I(G,R)$ is the direct sum of $\FF X_{I+J}X^-_{G-R-I}\cong\C_I(G,R)$ for all maximal $J$ in $\II(R-N(I))$.

\noindent(iv) 
The Cartan matrix of $\H(G,R)$ is the diagonal matrix ${\rm diag} \left\{ |\II(R-N(I))|: I\in\II(G-R) \right\}$.

\noindent(v) 
A complete set of primitive orthogonal idempotents of $H(G)$ is given by $\{X_IX^-_{G-R-I} : I\in\mathcal I(G-R)\}$.
\end{corollary}

\begin{proof}
(i) An algebra is semisimple if and only if its radical is 0. The radical of $\H(G,R)$ is generated by $\{x_r:r\in R\}$, which is 0 if and only if $R=\emptyset$.

(ii) There is a bilinear map 
$
\H(G,R)\times\C_I(G-R,\emptyset)\to  \P_I(G,R)
$
defined by sending $(X_J,z_I)$ to $X_J X_I X^-_{G-R-I}$ for all $J\in \II(G)$, where $z_I$ is an element spanning $\C_I(G-R,\emptyset)$. This induces an algebra surjection 
\[
\phi:\H(G,R)\otimes_{\H(G-R,\emptyset)} \C_I(G-R,\emptyset)\twoheadrightarrow \P_I(G,R)
\]
which sends $X_J\otimes_{\H(G-R,\emptyset)} z_I$ to $X_J X_I X^-_{G-R-I}$ for all $J\in\II(G)$. One sees that $\H(G,R)\otimes_{\H(G-R,\emptyset)}\C_I(G-R,\emptyset)$ is spanned by $\{X_J\otimes_{\H(G-R,\emptyset)} z_I:J\in\II(R-N(I))\}$, which is sent by $\phi$ to the basis (\ref{eq:BasisP}) for $\P_I(G,R)$. Hence $\phi$ must be an isomorphism.

(iii) If $J$ is maximal in $\II(R-N(I))$ then $\FF X_{I+J} X^-_{G-R-I}$ admits the same action of $\H(G,R)$ as $\C_I(G,R)$. Thus $\FF X_{I+J} X^-_{G-R-I}$ is a simple submodule of $\P_I(G,R)$ and must be contained in the socle of $\P_I(G,R)$. Conversely, we need to show that any simple submodule $M$ of $\P_I(G,R)$ is contained in the direct sum of $\FF X_{I+J} X^-_{G-R-I}$ for all maximal $J\in\II(R-N(I))$. Using the basis (\ref{eq:BasisP}) for $\P_I(G,R)$ one writes an arbitrary  element of $M$ as
\[
z = \sum_{J\in\II(R-N(I))} c_J X_{I+J}X^-_{G-R-I},\quad c_J\in\FF.
\]
Let $K$ be a minimal independent set in $\II(R-N(I))$ such that $c_K\ne0$. It suffices to show that $K$ is also maximal in $\II(R-N(I))$. If not, then there exists $r\in R-K$ such that $K+r\in\II(R-N(I))$. For any $J\in \II(R-N(I))$, one sees that 
\[
x_rX_{I+J}X^-_{G-R-I} = 
\begin{cases}
0, & {\rm if}\ r\in J\cup N(I\cup J), \\
X_{I+J+r}X^-_{G-R-I}\ne0, & {\rm otherwise.}
\end{cases}
\]
Thus in the expansion of $x_r z$ in terms of the basis (\ref{eq:BasisP}), the coefficients of $X_{I+K}X^-_{G-R-I}$ and $X_{I+K+r}X^-_{G-R-I}$ are $0$ and $c_K\ne0$, respectively. It follows that $x_rz\notin \FF z$ and $M$ is at least $2$-dimensional. This contradicts the simplicity of $M$.

(iv) Let $I\in \II(G-R)$. We order the elements $X_{I+ J} X^-_{G-R-I}$ by $|J|$ for all $J\in\II(R-N(I))$. This induces a filtration for $\P_I(G,R)$, under which 
\[
x_v X_{I+ J} X^-_{G-R-I} \equiv 
\begin{cases}
X_{I+ J} X^-_{G-R-I}, & v\in I,\\
0, & v\notin I.
\end{cases}
\]
Hence every simple composition factor of $\P_I(G,R)$ is isomorphic to $\C_I(G,R)$. The 
Cartan matrix follows.


(v) This follows from the decomposition of $\H(G,R)$ given in Theorem~\ref{thm:decomp} and the equality
\[
\sum_{I\in\mathcal I(G-R)} X_IX^-_{G-R-I} 
= \sum_{J\in\II(G-R)} \sum_{I\subseteq J} (-1)^{|J\setminus I|} X_J
= 1. 
\]
The reader who is not familiar with primitive orthogonal idempotents can find more details in~\cite[\S I.4]{ASS}.
\end{proof}

\subsection{Induction and restriction}
Let $G'$ be an induced subgraph of $G$ and let $R'=G'\cap R$. By Corollary~\ref{cor:EmbedAlgebra}, the following induction and restriction are well defined for isomorphism classes of modules: 
\begin{itemize}
\item
the induction $M\uparrow\,\!_{G',R'}^{G,R}:=\H(G,R)\otimes_{\H(G',R')} M$ of an $\H(G',R')$-module $M$ to $\H(G,R)$,\vskip3pt
\item
the restriction $N\downarrow\,\!_{G',R'}^{G,R}$ of an $\H(G,R)$-module $N$ to $\H(G',R')$. 
\end{itemize}

\begin{proposition}\label{prop:IndC}
Assume $R=\emptyset$, and hence $R'=\emptyset$. Write $(G,R)=(G)$ and $(G',R')=(G')$. Then for any $I'\in\II(G')$,
\[
\C_{I'}(G')\uparrow\,_{G'}^G \cong\bigoplus_{I\in\II(G):I\cap G'=I'} \C_I(G).
\]
\end{proposition}

\begin{proof}
Suppose that $\C_{I'}(G')=\FF z$. Using the universal property of the tensor product one obtains an algebra sujection
\[
\phi: \H(G)\otimes_{\H(G')}\FF z \twoheadrightarrow \H(G)X_{I'}X^-_{G'-I'}
\]
which sends  $X_J\otimes_{\H(G')} z$ to $X_J X_{I'}X^-_{G'-I'}$ for all $J\in\II(G)$. One sees that $\H(G)\otimes_{\H(G')}\FF z$ is spanned by
\[
\{X_I\otimes_{\H(G')} z: I\in\II(G),\ I\cap G'=I'\}
\]
since $x_vz=0$ for all $v\in G'-I'$. This spanning set is sent by $\phi$ to 
\[
\{X_I X^-_{G'-I'}: I\in\II(G),\ I\cap G'=I'\}
\]
which is a basis for $\H(G)X_{I'}X^-_{G'-I'}$ since it is a spanning set triangularly related to $\{X_I:I\in\II(G),\ I\cap G'=I'\}$, a linearly independent set in $\H(G)$. Thus $\phi$ is an isomorphism. Using the length filtration induced by $|I|$ for all $I$ appearing in the above basis, one sees that the composition factors of $\H(G)X_{I'}X^-_{G'-I'}$ are $\C_I(G)$ for all $I\in\II(G)$ with $I\cap G'=I'$, each appearing exactly once. This completes the proof as $\H(G)$ is semisimple by Corollary~\ref{cor:decomp} (i).
\end{proof}

\begin{proposition}\label{prop:IndP}
Let $I\in\II(G-R)$ and $J\in\II(G'-R')$. Then $\C_I(G,R)\downarrow\,\!_{G',R'}^{G,R} \cong \C_{I\cap G'}(G',R')$ and 
\[
\P_J(G',R')\uparrow\,_{G',R'}^{G,R} \cong \bigoplus_{K\in\II(G-R): K\cap G' = J } \P_K(G,R).
\]
\end{proposition}

\begin{proof}
The restriction of $\C_I(G,R)$ follows easily from the definition. By Corollary~\ref{cor:decomp} (ii) and Proposition~\ref{prop:IndC},
\begin{eqnarray*}
\P_J(G',R')\uparrow\,_{G',R'}^{G,R} &\cong&
\C_J(G'-R',\emptyset) \uparrow\,_{G'-R',\emptyset}^{G',R'} \uparrow\,_{G',R'}^{G,R} \\
&\cong& \C_J(G'-R',\emptyset) \uparrow\,_{G'-R',\emptyset}^{G,R} \\
&\cong& \C_J(G'-R',\emptyset) \uparrow\,_{G'-R',\emptyset}^{G-R,\emptyset} \uparrow\,_{G-R,\emptyset}^{G,R} \\
&\cong & \bigoplus_{ K\in\II(G-R),\ K\cap G'=J} \C_K(G-R,\emptyset) \uparrow\,_{G-R,\emptyset}^{G,R}  \\
&\cong & \bigoplus_{K\in\II(G-R),\ K\cap G'=J} \P_K(G,R). 
\end{eqnarray*}
This completes the proof. 
\end{proof}

\begin{remark}
It is not hard to obtain the simple composition factors of the induction of a simple $\H(G',R')$-module to $\H(G,R)$. But the restriction of a projective indecomposable $\H(G,R)$-module to $\H(G',R')$ is not always projective.
\end{remark}

\section{Commutative Hecke algebras of type A}\label{sec:H01}

We apply the previous results to commutative Hecke algebras of type A with independent parameters.
 
\subsection{Decomposition of Fibonacci numbers}

Let $(W,S)$ be the Coxeter system of type $A_n$ whose Coxeter diagram is the path $\xymatrix @C=8pt{ s_1 \ar@{-}[r] & s_2 \ar@{-}[r] & \cdots \ar@{-}[r] & s_n}$. We often identify $s_i$ with $i$ and write $\bq:=(q_1,\ldots,q_n)\in \FF^n$. Let $\H(\bq)$ be a collapse-free and commutative Hecke algebra of $(W,S)$ with independent parameters $\bq$. Then Theorem~\ref{thm:commutative} implies that either $q_i=0$ for all odd $i\in[n]$ and $q_i\ne0$ for all even $i\in[n]$, or the other way around. Proposition~\ref{prop:HGR} provides an algebra isomorphism $H(\bq)\cong \H(P_n,R)$, where $R:=\{i\in[n]:q_i=-1\}$. Note that the set $R$ obtained from $\H(\bq)$ depends on $\cha(\FF)$. For example, if $\bq=(1,0,1,0,1,\ldots)$ then $R=\emptyset$ and $\H(P_n,R)$ is semisimple if $\cha\FF\ne2$, but $R=\{1,3,5,\ldots\}$ and $\H(P_n,R)$ is not semisimple if $\cha(\FF)=2$. However, the algebra $\H(P_n,R)$ is defined for any subset $R\subseteq[n]$ and our results do not depend on $\cha(\FF)$. We first give decompositions of the Fibonacci numbers.

\begin{proposition}\label{prop:DimDecomp}
Let $R\subseteq[n]$. Then
\[
F_{n+2}=\sum_{I\in\II(P_n-R)}|\II(R-N(I))|.
\]
\end{proposition}

\begin{proof}
Let $G$ be a simple graph and let $R\subseteq V(G)$. By Proposition~\ref{prop:basisG}, the dimension of $\H(G,R)$ is $|\II(G)|$. By Theorem~\ref{thm:decomp}, $\H(G,R)$ is the direct sum of $\P_I(G,R)$ for all $I\in\II(G-R)$, and the dimension of each $\P_I(G,R)$ is $|\II(R-N(I))|$. Hence 
\[
|\II(G)|=\sum_{I\in\II(G-R)}|\II(R-N(I))|.
\]
Now take $G=P_n$. We know that $|\II(P_n)|=F_{n+2}$ by Corollary~\ref{cor:ComDim}. Thus the result holds.
\end{proof}

\begin{example}\label{example:Fib1}
Let $R:=[m]$ for some $m\in[n-1]$. Then the subgraph of $P_n$ induced by $R$ is the path $P_m$. If $I\in\II(P_n-[m+1])$ then $\II(R-N(I))=\II(R)$. If $I\in\II(P_n-R)$ contains $m+1$ then $I-\{m+1\}\in\II(P_n-[m+2])$ and $\II(R-N(I)=\II([m-1])$. Thus we recover a well known identity
$
F_{n+2} = F_{m+2}F_{n-m+1} + F_{m+1}F_{n-m}.
$
\end{example}

\begin{example}\label{example:Fib2}
Let $X$ and $Y$ be the subsets of odd and even numbers in $[n]$, respectively. Then 
\[
F_{n+2} = \sum_{ I\subseteq X} 2^{|Y-N(I)|} = \sum_{J\subseteq Y} 2^{|X-N(J)|}. 
\]
This writes a Fibonacci number as a sum of $2^{|X|}$ or $2^{|Y|}$ many powers of $2$. Some small examples are provided below.
\[
\begin{tabular}{|c|c||c|c|}
\hline
n=1 &  2 = 1+1 = 2 & n=2 & 3 = 2+1 \\
\hline
n=3 & 5 = 2+1+1+1 = 4+1 & n=4 & 8 = 4+2+1+1 \\
\hline
n=5 & 13 = 4+2+2+1+1+1+1=1 = 8+2+2+1 & n=6 & 21 = 8+4+2+2+2+1+1+1 \\
\hline
\end{tabular}
\]
\end{example}

\subsection{The semisimple commutative case}

Now we study the representation theory of the semisimple commutative algebra $\H_n := \H(P_{n-1},\emptyset)$, where $\H_0:=\FF$ by convention. We write $\alpha\propto n$ if $\alpha=(\alpha_1,\ldots,\alpha_\ell)$ is a composition of $n$ with all internal parts larger than $1$, i.e. $\alpha_i>1$ whenever $1<i<\ell$.

\begin{proposition}
The algebra $\H_n$ decomposes into a direct sum of $F_{n+1}$ many 1-dimensional simple submodules $\C_\alpha$ indexed by $\alpha\propto n$, with the $\H_n$-action on $\C_\alpha$ given by $x_i=1$ if $i\in D(\alpha)$ or $x_i=0$ otherwise.
\end{proposition} 

\begin{proof}
For any composition $\alpha$ of $n$, one sees that $D(\alpha)$ is an independent set of $P_{n-1}$ if and only if $\alpha$ has no internal parts equal to $1$. Thus the result follows from Theorem~\ref{thm:decomp}.
\end{proof}

Since $\H_n$ is semisimple, its two Grothendieck groups $G_0(\H_n)$ and $K_0(\H_n)$ are the same. Given nonnegative integers $m$ and $n$, the subalgebra of $\H_{m+n}$ generated by $x_1,\ldots,x_{m-1},x_{m+1},\ldots,x_{m+n-1}$ is isomorphic to $\H_m\otimes\H_n$, giving a natural embedding $\H_m\otimes\H_n\hookrightarrow\H_{m+n}$. Thus there is a tower $\H_\bullet: \H_0\hookrightarrow \H_1\hookrightarrow \H_2\hookrightarrow \cdots$ of algebras, whose Grothendieck group $G_0(\H_\bullet):=\bigoplus_{n\geq0} G_0(\H_n)$ has a product and a coproduct defined by
\[
\C_\alpha \htimes \C_\beta := \left( \C_\alpha \otimes \C_\beta \right) \uparrow\,\!_{\H_m\otimes \H_n}^{\H_{m+n}} 
\quad\text{and}\quad
\Delta(\C_\alpha):=\sum_{0\le i\le m} \C_\alpha \downarrow\,\!_{\H_i\otimes \H_{m-i}}^{\H_m}
\]
for all $\alpha\propto m$ and $\beta\propto n$. One sees that the product $\htimes$ and the coproduct $\Delta$ are well defined, with unit $u$ sending $1$ to $C_\emptyset$, and counit $\epsilon$ sending $C_\emptyset$ to $1$ and $C_\alpha$ to $0$ for all $\alpha\propto n$, $n\geq1$. Applying Proposition~\ref{prop:IndP} immediately gives the following explicit formulas for the product and coproduct below. See \S\ref{sec:rep} for the notation $\alpha\beta$, $\alpha\rhd\beta$, $\alpha_{\leq i}$, and $\alpha_{>i}$.

\begin{proposition}\label{prop:simple}
For any $\alpha\propto m$ and $\beta\propto n$, one has
\[
\C_\alpha \htimes \C_\beta 
= \begin{cases}
\C_{\alpha\beta} \oplus \C_{\alpha \rhd \beta}, & {\rm if}\ \alpha\beta\propto m+n, \\
\C_{\alpha\rhd\beta}, & {\rm otherwise,}
\end{cases}
\quad\text{and}\quad
\Delta(\C_\alpha) = \sum_{0\le i\le m} \C_{\alpha_{\leq i}}  \otimes \C_{\alpha_{>i}}.
\]
\end{proposition}

For example, one has $\C_{132}\htimes \C_{41} = \C_{13241}\oplus\C_{1361}$, $\C_{121}\htimes \C_{32} = \C_{1242}$, and
\[
\Delta(\C_{122}) = \C_\emptyset\otimes\C_{122} + \C_1\otimes\C_{22} + \C_{11}\otimes \C_{12} + \C_{12} \otimes C_2 + \C_{121}\otimes\C_1 + \C_{122}\otimes \C_\emptyset.
\]

\begin{corollary}\label{cor:duality}
(i) The graded algebra and coalgebra structures of $G_0(\H_\bullet)$ are dual to each other via the pairing defined by $\langle \C_\alpha,\C_\beta\rangle:=\delta_{\alpha,\,\beta}$ for all $\alpha\propto m$ and $\beta\propto n$, with a self-dual basis $\{\C_\alpha:\alpha\propto n,\ \forall n\geq0\}$.

\noindent
(ii) There is a surjection $\sigma:K_0(\H_\bullet(0)) \twoheadrightarrow G_0(\H_\bullet)$ of graded algebras and an injection $\iota: G_0(\H_\bullet)\hookrightarrow G_0(\H_\bullet(0))$ of graded coalgebras such that the two maps are dual to each other.
\end{corollary}
 
\begin{proof}
The first assertion holds since it follows from Proposition~\ref{prop:simple} that 
\begin{equation}\label{eq:duality}
\langle \C_\alpha \htimes \C_\beta, \C_\gamma \rangle = \langle \C_\alpha\otimes\C_\beta, \Delta(\C_\gamma) \rangle,\quad \langle \C_\emptyset, \C_\alpha \rangle = \epsilon(\C_{\alpha}).
\end{equation}
For the second assertion, first recall the representation theory of the 0-Hecke algebra $H_n(0)$ from \S\ref{sec:rep}. We define the surjection $\sigma$ by 
\begin{equation}\label{eq:surj}
\sigma(\P_\alpha(0))=
\begin{cases}
\C_\alpha, & \textrm{if } \alpha\propto n, \\
0, & \textrm{otherwise.}
\end{cases}
\end{equation}
We define the injection $\iota$ by sending $\C_\alpha$ to $\C_\alpha(0)$ for all $\alpha\propto n$. One sees that $\sigma$ and $\iota$ are maps of graded algebras and coalgebras, respectively, by comparing Proposition~\ref{prop:simple} with Proposition~\ref{prop:ProdP}. It is not hard to check that 
\[
\langle \sigma(\P_\alpha(0)),\C_\beta \rangle = \langle \P_\alpha(0), \iota(\C_\beta) \rangle = \delta_{\alpha,\,\beta},\quad \forall \alpha\models m,\ \forall \beta\propto n.
\]
This shows that $\sigma$ and $\iota$ are dual maps. Hence (ii) holds.
\end{proof}
 
\begin{remark}
(i) Comparing the definitions for $\H_n$ and $\H_n(0)$ one sees that the former is a quotient of the latter by the relations $T_iT_{i+1}=0$ for all $i=1,\ldots,n-2$. Thus any $\H_n$-module is automatically an $\H_n(0)$-module. This induces the injection $\iota:G_0(\H_\bullet)\hookrightarrow G_0(\H_\bullet(0))$ given in the previous proposition. On the other hand, $\C_\alpha(0)={\rm top} (\P_\alpha(0))$ admits an $\H_n$-action and is hence isomorphic to $\C_\alpha$ if and only if the composition $\alpha$ has all internal parts larger than $1$. This induces the surjection $\sigma:K_0(\H_\bullet(0))\twoheadrightarrow G_0(\H_\bullet)$ defined in (\ref{eq:surj}).

(ii) It is well known that the number of partitions of $n$ is no more than the Fibonacci number $F_{n+1}$. One may suspect that the surjection $K_0(\H_\bullet(0))\cong\NSym\twoheadrightarrow \Sym\cong G_0(\CC\SS_\bullet)$ factors through the surjection $\sigma: K_0(\H_\bullet(0))\twoheadrightarrow G_0(\H_\bullet)$. This is \emph{not} true since the commutative image of the noncommutative ribbon Schur function $\bs_\alpha$ is the ribbon schur function $s_\alpha$, but $f(\P_\alpha(0))=0$ if $\alpha$ is a composition with an internal part equal to $1$. Similarly, one sees that the injection $G_0(\CC\SS_\bullet)\cong\Sym\hookrightarrow \QSym\cong G_0(\H_\bullet(0))$ does not factor through the injection $\iota: G_0(\H_\bullet)\hookrightarrow G_0(\H_\bullet(0))$, since the image of the injection $i$ is spanned by $\C_\alpha(0)$ for all $\alpha\propto n$, $n\geq0$, but $F_\alpha\in \Sym$ when $\alpha=1^n$, $n\geq3$.

(iii) Unfortunately, $G_0(\H_\bullet)$ is not a bialgebra: one checks that $\Delta(\C_{11}\htimes\C_1)\ne \Delta(\C_{11})\htimes\Delta(\C_1)$ where the product on the right hand side is tensor-component-wise. Thus it does not fit into Zelevinsky's theory on \emph{positive self-dual Hopf algebras}~\cite{Zelevinsky}. One also checks that $G_0(\H_\bullet)$ is not a \emph{weak bialgebra} (c.f.~\cite{WeakHopf}), nor an \emph{infinitesimal bialgebra} (c.f.~\cite{InfinitesimalHopf}).
\end{remark}

Next we consider the \emph{Bratteli diagram} of the tower of algebras $\H_0\hookrightarrow \H_1\hookrightarrow \H_2\hookrightarrow \cdots$. It has vertices at level $n$ indexed by $\alpha\propto n$, for $n=0,1,2,\ldots$, and it has an edge between $\alpha\propto n$ and $\beta\propto n-1$ if and only if $\C_\alpha\downarrow\,_{\H_{n-1}}^{\H_n}\cong \C_{\beta}$. One can draw this diagram using  Proposition~\ref{prop:simple}.
The first $5$ levels are illustrated below.
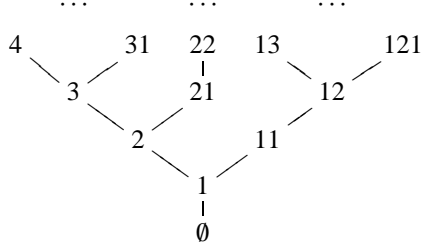
\[
\xymatrix @C=8pt @R=5pt {
&\cdots && \cdots && \cdots  \\
4 & & 31 & 22 & 13 & & 121 \\
& 3 \ar@{-}[lu] \ar@{-}[ru] & & 21 \ar@{-}[u] & & 12 \ar@{-}[lu] \ar@{-}[ru] \\
& & 2 \ar@{-}[lu] \ar@{-}[ru] & & 11 \ar@{-}[ru]\\
& & & 1 \ar@{-}[lu] \ar@{-}[ru] \\
& & & \emptyset \ar@{-}[u]
} \]

\subsection{Antipode}

We consider the antipode of $G_0(\H_\bullet)$. In general, let $A$ be an algebra with product $\mu$ and unit $u$, and let $C$ be a coalgebra with coproduct $\Delta$ and counit $\epsilon$. The \emph{convolution product} of two maps $f,g\in {\rm Hom}_{\,\FF}(C,A)$ is defined as $f\star g := \mu\circ(f\otimes g)\circ \Delta$. One can check that $u\circ \epsilon$ is the two-sided identity element for this convolution product. 

Let $(A',\mu',u')$ be another algebra and $(C',\Delta',\epsilon')$ be another coalgebra such that there exists an algebra surjection $\sigma: A\twoheadrightarrow A'$ and a coalgebra injection $\iota: C'\hookrightarrow C$. Then $u'=\sigma\circ u$, $\epsilon'=\epsilon\circ\iota$, and the following diagram is commutative, where $f':=\sigma\circ f\circ\iota$ and $g':=\sigma\circ g \circ \iota$.
\begin{equation}\label{eq:convolution}
\xymatrix @R=15pt{ 
C \ar@{->}[r]^-{\Delta} & C\otimes C \ar@{->}[r]^{f\otimes g} & A\otimes A \ar@{->>}[d]^{\sigma\otimes\sigma} \ar@{->}[r]^-{\mu} & A \ar@{->>}[d]^\sigma \\
C' \ar@{^(->}[u]^\iota \ar@{->}[r]^-{\Delta'} & C'\otimes C' \ar@{^(->}[u]^{\iota\otimes\iota} \ar@{->}[r]^{f'\otimes g'} & A'\otimes A' \ar@{->}[r]^-{\mu'} & A' 
}
\end{equation}

The \emph{antipode} $S$ of a Hopf algebra $H$ is nothing but the 2-sided inverse of the identity map $1_H$ under the convolution product for the endomorphism algebra ${\rm End}_{\,\FF}(H)$. In other words, $S$ is defined by the commutative diagram below.
\[
\xymatrix @R=10pt @C=16pt {
& H\otimes H \ar@{->}[rr]^{S\otimes 1_H} & & H\otimes H \ar@{->}[rd]^{\mu} \\
H \ar@{->}_{\Delta}[rd] \ar@{->}^{\Delta}[ru] \ar@{->}^{\epsilon}[rr] & & \FF \ar@{->}[rr]^{u} & & H \\
& H\otimes H \ar@{->}[rr]_{1_H\otimes S} & & H\otimes H \ar@{->}[ru]_{\mu}
}
\] 
Note that the definition for the antipode $S$ only requires $H$ to be simultaneously an algebra and a coalgebra. Moreover, if the antipode $S$ of $H$ exists, and if there is an algebra surjection $\sigma: H\twoheadrightarrow H'$ and a coalgebra injection $\iota: H'\hookrightarrow H$, then one sees from \eqref{eq:convolution} that $S':=\sigma\circ S\circ \iota$ is the antipode of $H'$.

The antipodes of the dual graded Hopf algebras $\QSym$ and $\NSym$ are well known to the experts. If $\alpha=(\alpha_1,\ldots,\alpha_\ell)$ is a composition of $n$ then its \emph{reverse} is the composition rev$(\alpha):=(\alpha_\ell,\ldots,\alpha_1)$ and its \emph{conjugate} is the composition $\omega(\alpha):=({\rm rev}(\alpha))^c = {\rm rev}(\alpha^c)$. For example, if $\alpha=21321$ then rev$(\alpha)=12312$ and $\omega(\alpha)=22131$. The antipodes of $\QSym$ and $\NSym$ are defined by $S(F_\alpha) = (-1)^n F_{\omega(\alpha)}$ and $S(\bs_\alpha) = (-1)^n \bs_{\omega(\alpha)}$ for all $\alpha\models n$, $n\ge0$, where $\{F_\alpha\}$ and $\{\bs_\alpha\}$ are dual bases for $\QSym$ and $\NSym$.

However, the same rule does not work for $G_0(\H_\bullet)$. To give the antipodes of $G_0(\H_\bullet)$ we introduce a free $\mathbb Z$-module $\Comp$ with a basis consisting of all compositions. By Proposition~\ref{prop:ProdP}, we can define a product $\alpha\htimes\beta := \alpha\beta+\alpha\rhd\beta$ and a coproduct $\Delta(\alpha) := \sum_{0\leq i\leq|\alpha|} \alpha_{\leq i}\otimes\alpha_{>i}$ for all compositions $\alpha$ and $\beta$, such that there is an algebra isomorphism $\Comp\cong K_0(\H_\bullet(0))$ and a coalgebra isomorphism $\Comp\cong G_0(\H_\bullet(0))$. The basis of all compositions for $\Comp$ is self-dual under the pairing $\langle \alpha,\beta\rangle:= \delta_{\alpha,\beta}$. There is an algebra surjection $\sigma: \Comp\twoheadrightarrow G_0(\H_\bullet)$  defined by 
\[
\sigma(\alpha) = 
\begin{cases}
\C_\alpha, & \alpha\propto n,\\
0, &  {\rm otherwise,}
\end{cases} 
\quad\forall \alpha\models n,\quad\forall n\ge0
\]
and a coalgebra injection $\iota: G_0(\H_\bullet)\hookrightarrow \Comp$ sending $\C_\alpha$ to $\alpha$ for all $\alpha\propto n$, $n\ge0$. They are dual to each other by Corollary~\ref{cor:duality} (ii). One can check that $\Comp$ is not a bialgebra, but its antipode exists, giving the antipode of $G_0(\H_\bullet)$.

\begin{proposition}
The map $S$ sending $\alpha$ to $(-1)^n \alpha^c$ for all $\alpha\models n$, $n\geq0$, is the antipode of $\Comp$. Consequently, the antipode of $G_0(\H_\bullet)$ is $\sigma\circ S\circ \iota$, which sends $\C_\alpha$ to $(-1)^n \C_{\alpha^c}$ if both $\alpha\propto n$ and $\alpha^c\propto n$ hold for some $n\ge0$, that is, if $\alpha\in\{22\cdots2, 122\cdots2, 22\cdots21, 122\cdots21\}$, or sends $\C_\alpha$ to $0$ otherwise.
\end{proposition}

\begin{proof}
If $S$ is the antipode of $\Comp$ then $\sigma\circ S\circ \iota$ is the antipode of $G_0(\H_\bullet)$. Thus it suffices to show that 
\[
\sum_{i=0}^n S(\alpha_{\leq i}) \htimes \alpha_{>i} 
 = u\circ \epsilon(\alpha) = \sum_{i=0}^n {\alpha_{\leq i}} \htimes S({\alpha_{>i}}),\quad \forall \alpha\models n.
\]
We only show the first equality and one can check that the same argument works for the second equality. It is trivial when $\alpha=\emptyset$. Assume $n\geq1$ below. Then $u\circ \epsilon(\alpha)=0$. For any $\beta\propto n$, it follows the self-duality of $\Comp$ that 
\begin{equation}\label{eq:antipode}
\left\langle \sum_{i=0}^n S({\alpha_{\leq i}}) \htimes {\alpha_{>i}}, \beta
\right\rangle 
= \sum_{i=0}^n  \left\langle S({\alpha_{\leq i}}) \otimes {\alpha_{>i}}, \Delta(\beta) \right\rangle
= \sum_{i=0}^n \langle S({\alpha_{\leq i}}), {\beta_{\leq i}} \rangle \cdot \langle {\alpha_{>i}}, {\beta_{>i}} \rangle.
\end{equation}
Thus it suffices to show that the sum of $L_i:=\langle S({\alpha_{\leq i}}), {\beta_{\leq i}} \rangle \cdot \langle {\alpha_{>i}}, {\beta_{>i}} \rangle$ for $i=0,1,\ldots,n$ equals $0$. One sees that
\[
L_i = \begin{cases}
(-1)^i, & {\rm if}\ (\alpha_{\leq i})^c = \beta_{\leq i},\ \alpha_{>i} = \beta_{>i},\\
0, & {\rm otherwise}.
\end{cases}
\]
Let $N$ be the set of all $i\in\{0,1,\ldots,n\}$ such that $L_i\ne0$. It is trivial if $N=\emptyset$. 

Suppose that $i\in N$. One sees that $D(\alpha_{\leq j})=D(\alpha)\cap[j-1]$ and $D(\alpha_{>j})=D(\alpha)\cap\{j+1,\ldots,n-1\}$ for any $j$; similarly for $\beta$. Hence $(\alpha_{\leq i})^c = \beta_{\leq i}$ implies $(\alpha_{\leq j}) = \beta_{\leq j}$ for all $j<i$, and $\alpha_{>i} = \beta_{>i}$ implies $\alpha_{>j} = \beta_{>j}$ for all $j>i$.

Since $(\alpha_{\leq i})^c = \beta_{\leq i}$, the number $i-1$ must belong to exactly one of $D(\alpha)$ and $D(\beta)$. This forces $\alpha_{>j} \ne \beta_{>j}$ for all $j<i-1$. Similarly, since $\alpha_{>i} = \beta_{>i}$, the number $i+1$ belongs to both or neither of $D(\alpha)$ and $D(\beta)$. This forces $(\alpha_{\leq j})^c \ne \beta_{\leq j}$ for all $j>i+1$. Hence $N\subseteq\{i-1,i,i+1\}$.

If $i$ belongs to exactly one of $D(\alpha)$ and $D(\beta)$, then $N=\{i,i+1\}$ since $(\alpha_{\leq i+1})^c = \beta_{\leq i+1}$ and $\alpha_{>i-1}\ne \beta_{i-1}$.

If $i$ belongs to both or neither of $D(\alpha)$ and $D(\beta)$, then $N=\{i-1,i\}$ since $(\alpha_{\leq i+1})^c \ne \beta_{\leq i+1}$ and $\alpha_{>i-1} = \beta_{i-1}$.

In either case above the equation (\ref{eq:antipode}) equals $1 - 1 = 0$. This completes the proof.
\end{proof}

\section{Questions and Remarks}\label{sec:future}

\subsection{Dimension}
If the Coxeter system $(W,S)$ is simply laced then using the basis for $\H(\bq)$ provided in  Theorem~\ref{thm:SimplyLaced} one can obtain recursive formulas for the dimension of $\H(\bq)$. Is there anything else (e.g. closed formula and combinatorial interpretation) one can say about this dimension?  More generally, how to write down a basis for $H(\bq)$ of an arbitrary Coxeter system? 

\subsection{Type A}
In type A we know that the dimension of a collapse-free and commutative $\H(\bq)$ is a Fibonacci number; for example, one can take $\bq=(0,1,0,1,\ldots)$ or $\bq=(1,0,1,0,\ldots)$. What if $\H(\bq)$ is not commutative? 

For instance, let $\bq$ be a sequence of $m-1$ zeros followed by $n-1$ ones. Then $\H(\bq)$ is a quotient of $H_m(0)\otimes \FF\SS_n$ and has dimension $(m-1)!(n!+m-1)$, by Theorem~\ref{thm:SimplyLaced}. How does the representation theory of this algebra connect to the representation theory of $H_m(0)$ and $\SS_n$?

Here is another example. If $\bq$ consists of $a$ many copies of $0$ followed by $b$ many copies of $q\ne0$ and then $c$ many copies of $0$, one can use Theorem~\ref{thm:SimplyLaced} to show that
\begin{eqnarray*}
\dim \H(\bq)
& = & c!(a!((b+1)!+a)+(a+1)!c).
\end{eqnarray*}
If $\bq$ consists of $a$ many copies of $q\ne0$ followed by $b$ many copies of $0$ and then $c$ many copies of $q'\ne0$, then 
\begin{eqnarray*}
\dim \H(\bq) 
&=& b!((a+1)!+b)+(b-1)!((a+1)!+b-1)((c+1)!-1).
\end{eqnarray*}
What is the representation theory of $\H(\bq)$ in these two cases?

A final remark for type A: the tower of algebras $\H_0\hookrightarrow \H_1\hookrightarrow \H_2\hookrightarrow \cdots$ are different from the tower of algebras defined by Okada~\cite{Okada}, whose dimensions are $n!$ and whose Bratteli diagram is the Young-Fibonacci poset. 

\subsection{Other types}
Our results on the commutative algebra $\H(G,R)$ applies to affine type A. Let $G$ be the cycle $C_n$ with vertices $1,\ldots,n$ and edges $\{1,2\},\ldots,\{n-1,n\},\{n,1\}$. We know that $\H(C_n,R)$ has a basis indexed by $\II(C_n)$. One checks that if $n\geq3$ then $\II(C_n)=\II(P_{n-1})\sqcup\II(P_{n-3})$, which is the shadow of the decomposition
\[
\H(C_n,R)\cong \H(P_{n-1},R\cap[n-1]) \oplus \H(P_{n-1},R\cap[n-1]) x_n.
\]
Hence for $n\geq 3$ one has $|\II(C_n)| = F_{n+1}+F_{n-1}=L_n$, where $L_n$ is the $n$-th \emph{Lucas number}. 
When $R=\emptyset$ the algebra $\H(C_n,\emptyset)$ is semisimple and has all simple modules 1-dimensional. Unfortunately, we do not have a tower of algebras $\H(C_n,\emptyset)$, since there is no natural embedding $C_n\hookrightarrow C_{n+1}$, and thus have no further result in this direction.

One can also take $G$ to be the Coxeter diagram of finite type $D_n$ ($n\geq2$) or affine type $\widetilde D_n$ ($n\geq 5$). The dimension of $\H(G,R)$ is $4,5,9,14,23,\ldots$ (OEIS entry A000285) or $17, 24, 41,65,106,\ldots$ (OEIS entry A190996) in these cases.

\subsection{Power series realization}
In Section~\ref{sec:H01} we defined an algebra and coalgebra structure for the Grothendieck group $G_0(\H_\bullet)$ of the tower of algebras $\H_\bullet: \H_0\hookrightarrow \H_1\hookrightarrow \H_2\hookrightarrow \cdots$, with a self-dual basis consisting of the simple modules, which are indexed by compositions with internal parts larger than 1. This is further extended to $\Comp$ with a basis indexed by all compositions. Is there a Frobenius type of characteristic map for $G_0(\H_\bullet)$, or in other words, is there a power series realization of $G_0(\H_\bullet)$ as both an algebra and a coalgebra, similarly to $G_0(\CC\SS_\bullet)\cong \Sym$, $G_0(\H_\bullet(0))\cong \QSym$, and $K_0(\H_\bullet(0))\cong \NSym$? And how about $\Comp$?





\end{document}